\documentclass[10pt,leqno,twoside]{amsart}

\usepackage{tikz, subfigure, xcolor} 
\usetikzlibrary{snakes}
\usepackage{amsmath}
\usepackage{amsthm}
\usepackage{amssymb}
\usepackage{amsfonts}
\usepackage{amsaddr}
\usepackage{cite}
\usepackage{dsfont}
\usepackage{hyperref}
\usepackage{inputenc}
\usepackage[normalem]{ulem}
\usepackage{doi}

\newtheorem{theorem}{Theorem}[section]
\newtheorem{lemma}[theorem]{Lemma}
\newtheorem{proposition}[theorem]{Proposition}
\newtheorem{corollary}[theorem]{Corollary}
\theoremstyle{definition}

\newtheorem{remark}[theorem]{Remark}

\numberwithin{equation}{section}

\newcommand{\divh}{\mathrm{div}_{\!H}\,}		
\newcommand{\nablah}{\nabla_{\!H}\,} 			
\newcommand{\Deltah}{\Delta_{\!H}\,}  			
\DeclareMathOperator{\nablatoalpha}{\nabla^{\alpha}}
\DeclareMathOperator{\nablatoalphastr}{\nabla^{\alpha'}}
\DeclareMathOperator{\nablatoalphaminstr}{\nabla^{\alpha-\alpha'}}

\DeclareMathOperator{\per}{\mbox{\tiny per}}

\providecommand{\norm}[1]{\left\| #1 \right\|} 

\renewcommand{\d}{\operatorname{d}\!} 
\DeclareMathOperator{\dz}{\d z}		
\DeclareMathOperator{\dt}{\d t}		
\DeclareMathOperator{\dr}{\d r} 	
\DeclareMathOperator{\dxi}{\d \xi}	
\DeclareMathOperator{\ddt}{\frac{\d}{\dt}} 

\newcommand{\R}{\mathbb{R}}

\newcommand{\N}{\mathbb{N}}

\title{Primitive Equations with half horizontal viscosity}


\author{Martin Saal}
\address{Department of Mathematics,
TU Darmstadt,\\ Schlossgartenstr.~7,
64289 Darmstadt, Germany}
\email{msaal@mathematik.tu-darmstadt.de}

\begin{document}

\begin{abstract}
We consider the $3D$ primitive equations and show, that one does need
less than horizontal viscosity to obtain a well-posedness result in 
Sobolev spaces. Furthermore, we will also investigate the primitive 
equations with horizontal viscosity and show that these equations are 
well-posed without imposing any boundary condition for the horizontal 
velocity components on the vertical boundary.
\end{abstract}

\maketitle


\section{Introduction}
The primitive equations are one of the fundamental models 
for geophysical flows and they are used to describe 
oceanic and atmospheric dynamics. They are derived from the
Navier-Stokes equations in domains where the vertical scale
is much smaller than the horizontal scale by performing the 
formal small aspect ratio limit. They  
describe the velocity $u$ of a fluid and the pressure $p$.
Putting $u = (v, w)$, where $v = (v_1, v_2)$ denotes the
horizontal components and $w$ stands for the vertical one,
the equations read with full viscosity
\begin{align*}
\left\{
\begin{array}{rll}
  \partial_t v + v \cdot \nablah v + w \partial_z v
  - \nu_1 \Deltah v-\nu_2 \partial_{zz}v + \nablah p & = 0, & \text{ in } \Omega
  \times (0, T),\\
   \partial_z p & =0, &\text{ in } \Omega \times (0, T), \\
  \divh v + \partial_z w & = 0, & \text{ in } \Omega \times
  (0, T), \\
  v(t=0) & = v_0, & \text{ in } \Omega.
\end{array}\right.
\end{align*}
Here $\Omega := G \times (-h, h) \subset \R^3$
for $h > 0$, $G \subset \R^2$; $\nu_1$ stands for the 
horizontal viscosity and $\nu_2$ for the 
vertical one; 
$\nablah$, $\divh$ and $\Deltah$ denote the horizontal
gradient, divergence and Laplacian, respectively and $v \cdot \nablah=v_1
\partial_x+v_2\partial_y$. Throughout this work we take $G=(-1,1)^2$. A rigorous 
justification of the small aspect ratio limit of the Navier-Stokes equations 
to the primitive equations is given in \cite{LiTiti2017}.
For simplicity we formulated the equations without the Coriolis force, 
but being a zero order term it does not change the qualitative results 
obtained by us.

Note, that the vertical velocity $w$ is determined by the 
divergence free condition and boundary conditions for $w$ 
on the bottom of the domain, 
so it has less regularity 
than $v$ making the nonlinear term $w\partial_z v$ stronger 
compared to the nonlinearity of the Navier-Stokes equation.

The mathematical 
analysis of the primitive equations has been started by Lions, 
Temam and Wang \cite{Lionsetal1992, Lionsetal1992_b, Lionsetal1993},
and in difference to the $3D$ Navier-Stokes equations 
the primitive equations are known to be time-global well-posed 
for initial data in $H^1(\Omega)$. This break through result 
has been proven by Cao and Titi \cite{CaoTiti2007} (see also 
Kobelkov \cite{Kobelkov2006} and Kukavica and Ziane \cite{Ziane2007}) 
and launched a lot of activity in the analysis of those equations.
In \cite{GigaGriesHusseinHieberKashiwabara2020}, Giga, 
Gries, Hieber, Hussein and Kashiwabara give results on 
strong well-posedness for only bounded initial data without 
any differentiability condition.
For more information on previous results on the primitive
equations we refer to the works of Washington and
Parkinson \cite{WashingtonParkinson1986}, Pedlosky
\cite{Pedlosky1987}, Majda \cite{Majda2003} and Vallis
\cite{Vallis2006}; see also the survey by Li and Titi 
\cite{LiTiti2016} for further references.

In the case $\nu_1=\nu_2=0$, i.e. when there is no viscosity at 
all, one obtains the $3D$ primitive Euler equations, 
which are also called the hydrostatic Euler equations.
The only existence result for this system is 
due to Kukavica, Temam, Vicol and Ziane
\cite{Kukavica2011}. They show that real-analytic 
data leads to a real-analytic local in time solution. 
For the $2D$ hydrostatic Euler equations Han-Kwan 
and Nguyen have shown in \cite{HanNguyen2016}, 
that this set of equations is ill-posed in the 
Sobolev space setting, in the sense that - without 
any additional condition on the data - the 
solution map cannot be H\"older continuous. 
Such an additional condition was first used by Brenier 
\cite{Brenier} and later by Masmoudi and Wong 
\cite{MasmoudiWong2012} to show 
the local in time well-posedness of the $2D$ 
hydrostatic Euler (see also \cite{Kukavica2014} for a generalization). 
They assume, that the initial 
velocity has a convex profile in the vertical direction 
which means, that one has a Rayleigh condition of the 
form $\partial_{zz}v\neq 0$ in $\Omega$. Regarding 
the question of global well-posedness, Cao et 
al \cite{Caoetal2015} and Wong \cite{Wong2015} have shown, 
that smooth solutions to the primitive Euler equations 
blow-up in finite time. In contrast to this situation, 
for the system with horizontal 
viscosity, i.e $\nu_1>0, \nu_2=0$ Cao, Li and Titi \cite{CaoLiTiti2016}
prove not only local but even 
global well-posedness for initial data in $H^2(\Omega)$.

So the question is natural, how much anisotropic viscosity is 
needed to obtain a local well-posedness result in 
Sobolev spaces. We will show, that one can take out at least 
"half" the horizontal viscosity by considering the following system.
\begin{align}\label{eq:primequhalfhorvisc}
\begin{split}
\partial_t v + v \cdot \nablah  v + w \partial_z v - A_{\perp} v + \nablah  p  & = 0 \qquad\mbox{in }(0,T)\times\Omega,  \\
\partial_z p  & = 0 \qquad\mbox{in }(0,T)\times\Omega,  \\
\divh v + \partial_z w & = 0 \qquad\mbox{in }(0,T)\times \Omega,\\
v(t=0)& = v_0  \!\!\qquad\mbox{in }\Omega
\end{split}
\end{align}
with periodic boundary conditions imposed on $v$ and $p$ in the horizontal directions, $w(z=\pm h)=0$ and 
\begin{align*}
 A_{\perp} =\left(\begin{matrix} \partial_{yy} & 0 \\ 0 & \partial_{xx}\end{matrix}\right).
\end{align*}
Note, that we do not have any boundary condition 
for $v$ on the vertical boundary. The 
main difficulty arises due to the nonlinear term 
$w\partial_zv$, because of the lack of 
regularity of $w$. In \cite{MasmoudiWong2012} Masmoudi and Wong 
make use of a special cancellation 
property related to this term when considering 
the $2D$ hydrostatic Euler equations to deduce bounds 
for $v_z$ under the condition $\partial_{zz}v\neq 0$. 
While in the $2D$ case this is sufficient to control 
also $v$, in the $3D$ setting this is not immediately 
possible. To work around that problem we combine their 
method with the approach Cao and Titi used
in \cite{CaoTiti2007} to split the function $v$ into 
two parts, the vertical average $\overline v$ 
(also called the barotropic mode) and an average 
free remainder $\tilde v$ (the baroclinic mode). 
Then $\overline v$ is the solution to a $2D$ equation 
containing the pressure and $\tilde v$ can 
be controlled by $v_z=\tilde v_z$. 
Our result than reads as follows.
\begin{theorem}\label{th:introd2}
Let $s\geq3$. Then for any horizontally periodic 
$v_0=(v_{01},v_{02})  \in  H^s(\Omega)$  
with $\partial_z v_0  \in  H^s(\Omega)$, 
$\int_{-h}^h\divh v_0(x,y,\xi)\dxi=0$  and 
$\partial_{zz} v_{0i}\neq 0$ in $\Omega$ for $i=1,2$
there exists a time $T>0$ and a unique strong solution 
$v$ to (\ref{eq:primequhalfhorvisc}) with
\begin{align*}
 & v \in L^{\infty}((0,T), H^{s}(\Omega)) \cap C^0([0,T], H^{s-\kappa}(\Omega)),\\
 & \partial_x v_2,\partial_y v_1 \in L^2((0,T),H^{s}(\Omega))
\end{align*}
for all $\kappa\in(0,1)$, and $\partial_z v$ has the same regularity as $v$.
\end{theorem}
Furthermore, we show a second possibility to take 
out half the horizontal viscosity by 
replacing $A_{\perp}$ in (\ref{eq:primequhalfhorvisc}) with
\begin{align*}
 A_{||}=\left(\begin{matrix} \partial_{xx} & 0 \\ 0 & \partial_{yy}\end{matrix}\right).
\end{align*}
This case is easier than the previous one, 
because the operator $A_{||}$ controls the horizontal 
divergence of $v$and thus also the function $w$ 
directly. We will prove the following well-posedness 
result for which no Rayleigh condition is needed.  
\begin{theorem}\label{th:introd3}
Let $s\geq3$. Then for any horizontally periodic $v_0\in  H^s(\Omega)$  with $\int_{-h}^h\divh v_0(x,y,\xi)\dxi=0$
there exists a time $T>0$ and a unique strong solution 
$v$ to (\ref{eq:primequhalfhorvisc}) with $A_{\perp}$ 
replaced by $A_{||}$ and we have
\begin{align*}
 & v \in L^{\infty}((0,T), H^{s}(\Omega)) \cap C^0([0,T], H^{s-\kappa}(\Omega)),\\
 & \partial_x v_1,\partial_y v_2 \in L^2((0,T),H^{s}(\Omega))
\end{align*}
for all $\kappa\in(0,1)$.
\end{theorem}

We will need the well-posedness of the primitive 
equations with only horizontal viscosity in the 
proofs of the Theorems (\ref{th:introd2}) and 
(\ref{th:introd3}), but that system is also of 
interest on its own. Due to turbulent mixing in 
the horizontal plane even in the case of full viscosities, 
the horizontal viscosity 
$\nu_1$ is much stronger than the vertical $\nu_2$. 
In the limiting case this means, that we have no 
vertical viscosity and we will consider that system,
\begin{align}\label{eq:primequhorvisc}
\begin{split}
\partial_t v + v \cdot \nablah  v + w \partial_z v - \Deltah v + \nablah  p  & = 0 \qquad\mbox{in }(0,T)\times\Omega,  \\
\partial_z p  & = 0 \qquad\mbox{in }(0,T)\times\Omega,  \\
\divh v + \partial_z w & = 0 \qquad\mbox{in }(0,T)\times \Omega,\\
v(t=0)& = v_0  \!\!\qquad\mbox{in }\Omega
\end{split}
\end{align}
with periodic boundary conditions imposed 
on $v$ and $p$ in the horizontal directions 
and $w(z=\pm h)=0$. Here we set $\nu_1=1$ for simplicity.
The first results for the primitive equations with only 
horizontal viscosity were obtained by Cao, Li and Titi 
 \cite{CaoLiTiti2016} and further investigated by them 
in \cite{CaoLiTiti2017}. In \cite{CaoLiTiti2016} they show 
the global well-posedness for initial data in $H^2(\Omega)$ 
additionally assuming homogeneous Neumann boundary conditions 
for $v$ on  the top $G\times \{h\}$ and the bottom 
$G\times \{-h\}$ of the domain. Their approach is to 
consider the case of full viscosity for which the global 
existence of solutions is known and to derive a priori bounds 
independent of the vertical viscosity $\nu_2$. For this 
boundary conditions there is no formation of a boundary layer 
and by letting $\nu_2$ tend to zero they obtain a solution 
to the primitive equations with only horizontal viscosity. However, 
not all solutions of (\ref{eq:primequhorvisc}) can be found by that 
method. We show, that the well-posedness result holds also without 
the Neumann conditions for $v$ on the top and on the bottom. Our 
approach is to interpret the term $w\partial v_z$ as a transport 
term with a non-constant coefficient $w$, which vanishes on the 
boundary. So there is no transport through the boundary and thus 
we do not need any condition on $v$ there and we obtain 
the following result.
\begin{theorem}\label{th:introd1}
Let $s\geq 2$. Then for any horizontally periodic $v_0\in H^s(\Omega)$ with $\int_{-h}^h\divh v_0(x,y,\xi)\dxi=0$ 
there exists a time $T>0$ and a unique strong solution $v$ to (\ref{eq:primequhorvisc}) with
\begin{align*}
 v & \in L^{\infty}((0,T), H^{s}(\Omega)) \cap C^0([0,T], H^{s-\kappa}(\Omega))\\
\partial_x v,\partial_y v & \in L^2((0,T),H^{s}(\Omega))
\end{align*}
for all $\kappa\in(0,1)$. For $s=2$ this solution extends globally in time.
\end{theorem}
We will prove the local existence result in detail, 
the global existence then follows from the 
estimates proven in \cite{CaoLiTiti2016} 
for initial data in $H^2(\Omega)$.
Although no boundary conditions for $v$ on 
the top and the bottom are not needed 
for the well-posedness, homogeneous Neumann 
boundary conditions are preserved in time by 
the equations if they hold for the initial value.

This paper is organized as follows. In 
section 2 we list the most important definitions 
and notations and we reformulate the equation 
(\ref{eq:primequhorvisc}) for $v$ into a system 
of equations for the mean value $\overline v$ 
and the remainder $\tilde v$.
In section 3 we show the well-posedness of a 
linearized version of the primitive equations with 
horizontal viscosity by a Galerkin-approach, and in 
section 4 we use that result to prove Theorem (\ref{th:introd1}).
In section 5 we turn to the case of half horizontal 
viscosity and give the proofs for 
the Theorems (\ref{th:introd2}) and (\ref{th:introd3}).

\section{Notations and basic Lemmas}  
By $L^2(\Omega), L^2(G)$ we denote the standard Lebesgue spaces with the scalar products
\begin{align*}
\langle f,g \rangle_{\Omega}:= \int_{\Omega}f(x,y,z)g(x,y,z)\; \operatorname d (x,y,z)
\end{align*}
and $\langle f,g \rangle_{G}$ defined analogously, by $\norm{f}_{L^2(\Omega)}$ and $\norm{f}_{L^2(G)}$ 
we denote the induced norm.  We drop $\Omega$ and $G$ in the notation if the dependence is obvious.

For a function $f\in L^{\infty}((0,T), L^{\infty}(\Omega))$ we use the abbreviation
\begin{align*}
\norm{f}_{\infty}:=\sup_{t\in(0,T)}\norm{f(t)}_{L^{\infty}}.
\end{align*}
We write $v(t=0)$ for the function $v|_{t=0}$ and $v(z=\pm h)$ as a short form of $v|_{z=h}$ and $v|_{z=-h}$.

For $s\in\N$ the space $H^s(\Omega)$ consists of $f\in L^2(\Omega)$ such that $\nablatoalpha f \in L^2(\Omega)$ for $|\alpha|\leq s$ endowed with the norm
\begin{align*}
\norm{f}_{H^s(\Omega)}=\sum_{|\alpha|\leq s} \norm{\nablatoalpha f}_{L^2(\Omega)}.
\end{align*}
Here we used the multi-index notation 
$\nablatoalpha=\partial_x^{\alpha_1} 
\partial_y^{\alpha_2}\partial_z^{\alpha_3}$ 
for $\alpha\in\N_0^3$. The spaces $H^s(G)$ 
are defined analogously, and we will again 
just write $\norm{f}_{H^s}$ if there is no 
ambiguity. If $s\notin\N$ the spaces $H^s(\Omega), H^s(G)$ 
are defined by complex interpolation, see \cite{Triebel} for details.

To handle the periodic boundary conditions 
in the horizontal variables we set
\begin{align*}
H^s_{\per}(\Omega):=\{f\in H^s(\Omega)|f \mbox{ is periodic of order }s-1\mbox{ on }\partial G\times (-h,h)\}
\end{align*}
and
\begin{align*}
C^{\infty}_{\per}(\Omega):=\{f\in C^{\infty}(\Omega)|f \mbox{ is periodic of arbitrary order on }\partial G\times (-h,h)\}.
\end{align*}
It is easy to see that $H^s_{\per}(\Omega)$ 
equipped with the $H^s$-norm is a Banach space 
and that $C^{\infty}_{\per}(\Omega)$ is a dense subset.

When investigating the case of half horizontal 
viscosity we need furthermore a subset of 
$H^s(\Omega)$, which reflects the Rayleigh condition 
mentioned in the introduction. For $s\geq 3$ and $\eta>1$ 
we define 
\begin{align*}
H^s_{\per,\eta}(\Omega):=\left\{f\in H^s_{\per}(\Omega)\left|\frac{1}{\eta}\leq \frac{1}{|\partial_z f(x,y,z)|}\leq \eta\right.\right\}
\end{align*}
with
\begin{align*}
\norm{f}_{H^s_{\eta}}^2=\norm{f}^2_{H^{s-1}}+\norm{\partial_z^s f}^2_{L^2}+\sum_{|\alpha|=s,\alpha_3=0} \norm{\frac{\nablatoalpha f}{\sqrt{|\partial_{z}f|}}}_{L^2}^2.
\end{align*}
For $f\in H^s_{\per,\eta}(\Omega)$ this expression is equivalent to the $H^s(\Omega)$-norm. 

The following inequalities will be helpful 
when we show a-priori estimates for the 
solutions to the primitive equations.  
\begin{lemma}\label{lemma:lowreg}\ \\
a) Let $f,h,\nablah h \in L^2(\Omega)$ and $g\in H^1(\Omega)$, then
\begin{align*}
|\langle fg,h \rangle| & \leq c \norm{f}_{L^2(\Omega)}\norm{g}_{H^1(\Omega)}\left( \norm{\nablah h}_{L^2(\Omega)}^{1/2}\norm{h}_{L^2(\Omega)}^{1/2} +\norm{h}_{L^2(\Omega)}\right).
\end{align*}
b) Let $f  \in H^2(\Omega)$, then
\begin{align*}
\norm{f}_{L^{\infty}(\Omega)} & \leq c \norm{f}_{H^2(\Omega)}^{3/4} \norm{f}_{L^2(\Omega)}^{1/4}.
\end{align*}
\end{lemma}

\begin{proof}
a) Consider
\begin{align*}
|\langle fg,h \rangle_{\Omega}|	
								& \leq \int_{-h}^h\norm{f(z) g(z) h(z)}_{L^1(G)}\dz\\
								& \leq \norm{g}_{L^{\infty}((-h,h),L^4(G))}\int_{-h}^h \norm{f(z)}_{L^2(G)} \norm{h(z)}_{L^4(G)} \dz\\
								& \leq \norm{g}_{L^{\infty}((-h,h),L^4(G))}\norm{f}_{L^2(\Omega)}\norm{h}_{L^2((-h,h),L^4(G))}.
\end{align*}
The Gagliardo-Nirenberg inequality 
\begin{align*}
\norm{h(z)}_{L^4(G)}\leq c(\norm{\nablah h(z)}^{1/2}_{L^2(G)}\norm{h(z)}^{1/2}_{L^2(G)}+\norm{h(z)}_{L^2(G)})
\end{align*}
gives
\begin{align*}
\norm{h}_{L^2((-h,h),L^4(G))}^2 
								& \leq c \int_{-h}^h\norm{\nablah h(z)}_{L^2(G)}\norm{h(z)}_{L^2(G)}+\norm{h(z)}^2_{L^2(G)}\dz\\
								& \leq  c\norm{\nablah h}_{L^2(\Omega)}\norm{h}_{L^2(\Omega)}+c\norm{h}^2_{L^2(\Omega)}.
\end{align*}
In \cite[Lemma 2.3]{CaoLiTiti2017} it has been shown that
\begin{align*}
\norm{g}_{L^{\infty}((-h,h),L^4(G))} \leq c \ (\norm{g}_{L^2(\Omega)} +\norm{\partial_z g}_{L^2(\Omega)})^{1/2}(\norm{g}_{L^2(\Omega)} +\norm{\nablah g}_{L^2(\Omega)})^{1/2}
\end{align*}
and thus
\begin{align*}
|\langle fg,h \rangle|	
\leq c \norm{g}_{H^1(\Omega)} \norm{f}_{L^2(\Omega)}\left( \norm{\nablah h}_{L^2(\Omega)}^{1/2}\norm{h}_{L^2(\Omega)}^{1/2}+ \norm{h}_{L^2(\Omega)}\right).
\end{align*}
b) This is the well-known Agmon's inequality.
\end{proof}

The following Aubin-Lions Lemma is 
shown in \cite[Corollary 4]{Simon1987}.
\begin{lemma}\label{le:aubinlions}
Let $T>0$ and $X,Y$ and $Z$ be Banach 
spaces such that $X$ is compactly embedded 
in $Y$, and $Y$ is continuously embedded in $Z$.
\begin{itemize}
 \item[(i)] If $(f_n)_n\subset L^2((0,T),X)$ is 
 bounded and  $(\partial_t f_n)_n$ is bounded in 
 $L^2((0,T),Z)$ then there exists an in $L^2((0,T),Y)$ 
 convergent subsequence.
 \item[(ii)] If $(f_n)_n\subset L^{\infty}((0,T),X)$ 
 is bounded and  $(\partial_t f_n)_n$ is bounded in 
 $L^2((0,T),Z)$ then there exists an in $C^0([0,T],Y)$ 
 convergent subsequence.
\end{itemize}

\end{lemma}

Now let us reformulate the primitive 
equations (\ref{eq:primequhorvisc}) where 
we replace the operator $\Deltah$ by any 
constant coefficient operator $A$ acting 
only in the horizontal directions. It is 
easy to see that the divergence free 
condition $\partial_z w+\divh v=0$ and the 
boundary condition $w(z=\pm h)=0$ are equivalent to
\begin{align*}
w(t,x,y,z) =-\divh \int_{-h}^{z}v(t,x,y,\xi) \dxi \qquad \mbox{and} \qquad
 \divh \int_{-h}^h v(t,x,y,\xi)\dxi=0.
\end{align*}
This means, that the mean value of $v$ in the vertical direction
\begin{align*}
\overline v(t,x,y):=\frac{1}{2h}\int_{-h}^h v(t,x,y,\xi)  \dxi
\end{align*}
is divergence free. The pressure is constant 
in the vertical direction, and thus $\overline p = p$. 
For $\overline v, p$ and the remainder
\begin{align*}
\tilde v=v-\overline v
\end{align*}
we obtain the system of coupled equations 
\begin{align}\label{eq:primequhorviscbar}
\begin{split}
\partial_t \overline v+ \overline v \cdot \nablah  \overline v-A \overline v + \nablah  p  & =-K(\tilde  v) , \\
\divh \overline  v & = 0,\\
\overline v(t=0)& =\overline v_0  
\end{split}
\end{align}
and
\begin{align}\label{eq:primequhorvisctilde}
\begin{split}
\partial_t \tilde v + \tilde v \cdot \nablah  \tilde v & + \overline v \cdot \nablah  \tilde v+ \tilde v \cdot \nablah  \overline v + w \partial_z \tilde v - A \tilde v  = K(\tilde v),\\
w(t,x,y,z) & =-\divh \int_{-h}^{z} \tilde v(t,x,y,\xi) \dxi,\\
 \tilde v(t=0)  & =\tilde v_0  
\end{split} 
\end{align}
both with periodic boundary conditions in the horizontal directions and where the coupling term $K(\tilde v)$ is given by
\begin{align}\label{eq:primequcoupling}
\begin{split}
& K(\tilde v)(t,x,y) \\
&\quad  =\frac{1}{2h}\int_{-h}^h \tilde v(t,x,y,\xi) \cdot \nablah  \tilde v(t,x,y,\xi)+  \tilde v(t,x,y,\xi)\; \divh \tilde v(t,x,y,\xi)  \dxi.
\end{split}
\end{align}
Therefore, the well-posedness of the primitive equations (\ref{eq:primequhorvisc}) with periodic boundary conditions in the horizontal directions and $w(z=\pm h)=0$ is equivalent to the well-posedness of (\ref{eq:primequhorviscbar})-(\ref{eq:primequcoupling}) with periodic boundary conditions in the horizontal directions, and the same holds for the cases $A=A_{\perp}$ and $A=A_{||}$.

\section{A linearized equation}
For given functions $w,a,b$ and $f$ and initial data $v_0$ we consider the equation
\begin{align}\label{eq:primequlin}
\begin{split}
\partial_t v+av+ b\cdot\nablah v+ w \partial_z v - \Deltah v  & = f \text{ in } (0,T)\times\Omega,  \\
v(t=0) & =v_0 \text{ in } \Omega
\end{split}
\end{align}
with periodic boundary conditions in the horizontal directions and show the existence of a solution by a Galerkin-approach. Note that if $w(z=\pm h)=0$ we do not need any boundary condition for $v$ in the vertical direction.

This is a linearized version of (\ref{eq:primequhorvisctilde}) for $A=\Deltah$, and based on its well-posedness we will show the local existence of solutions to the primitive equations with horizontal viscosity.

We work in the spaces
\begin{align*}
 H:=\{v\in L^2(\Omega)| \partial_z v\in L^2(\Omega)\}=H^1((-h,h),L^2(G))
\end{align*}
equipped with the scalar product $\langle u,v\rangle_H:=\langle u,v\rangle_{\Omega}+\langle \partial_z u,\partial_z v\rangle_{\Omega}$ and
\begin{align*}
 V:=\{v\in H| \nablah v\in H\}\cap H^1_{\per}(\Omega)=H^1((-h,h),H^1(G))\cap H^1_{\per}(\Omega)
\end{align*}
with the scalar product $\langle u,v\rangle_V:=\langle u,v\rangle_{H}+\langle \nablah u,\nablah v\rangle_{H}$. By $V'$ we denote the dual space of $V$, i.e., 
\begin{align*}
 V'=H^1((-h,h),H^{-1}(G)).
\end{align*}
We will also denote the dual pairing in $V\times V'$ by
$\langle \cdot ,\cdot\rangle_H$ to keep the notation simple.
For $g,h\in L^2((0,T),L^2(\Omega))$ we denote by
\begin{align*}
\langle g,h \rangle_{T}:=\int_{0}^T \int_{\Omega}g(t,x,y,z)h(t,x,y,z)\; \operatorname d (x,y,z)\dt
\end{align*}
the scalar product in space and time.

The solution we obtain in the first step 
will be a weak solution, where weak means 
"weak with respect to $x,y$", i.e., that for 
all $\varphi\in C_c^{\infty}((0,T),L^2((-h,h),H^1_{\per}(G)))$
\begin{multline*}
-\langle  v,\partial_t \varphi\rangle_T+\langle  av, \varphi\rangle_T+\langle   b\cdot\nablah v, \varphi\rangle_T+\langle w \partial_z v,\varphi\rangle_T + \langle \nablah  v,\nablah \varphi\rangle_T  =\langle f ,\varphi\rangle_T.
\end{multline*}
With this notion of solution we then 
have the following existence result.
\begin{theorem}\label{th:linex}
Let $v_0\in H$, $w,w_z,a,a_z,b=(b_1,b_2),b_z \in L^{\infty}((0,T)\times\Omega)$ with $w(z=\pm h)=0$ and $f\in L^2((0,T),V')$. Then there is a unique weak solution
\begin{align*}
v\in L^2((0,T),V) \cap C^0([0,T], H)\cap H^1((0,T),L^2((-h,h),H^{-1}(G)))
\end{align*}
to (\ref{eq:primequlin}) with
\begin{multline*}
\|v\|_{L^{\infty}((0,T),H)}^2+\|v\|_{L^2((0,T),V)}^2\\
\qquad \leq \left(\|v_0\|_{H}^2+2\|f\|_{L^2((0,T),V')}^2 \right) e^{\left(\frac12+\|w_z\|_{\infty}+2\|(a,a_z)\|_{\infty}+2\|(b,b_z)\|_{\infty}^2\right) T}.
\end{multline*}
\end{theorem}

\begin{proof}
Uniqueness: Let $v_1,v_2$ be weak solutions 
to the same initial data. Then the difference 
$v:=v_1-v_2$ solves
\begin{align*}
\partial_t v+av+ b\cdot\nablah v+ w \partial_z v - \Deltah v  = 0, \\
v(t=0)=0.
\end{align*}
$v$ is regular enough to test this equation 
with itself in $L^2(\Omega)$, which yields
 \begin{align*}
\frac12 \partial_t \|v\|_{L^2}^2+ \| \nablah v \|_{L^2}^2 
& =- \langle  av, v\rangle_{\Omega}-\langle   b\cdot\nablah v, v\rangle_{\Omega}+\frac12 \langle w_z \cdot  v,v \rangle_{\Omega} \\
& \leq \|a\|_{\infty}\|v\|_{L^2}^2+\|b\|_{\infty}\|\nablah v\|_{L^2}\|v\|_{L^2} + \frac12 \|w_z\|_{\infty}  \|v\|_{L^2}^2\\
& \leq \left(\|a\|_{\infty}+\frac12 \|b\|^2_{\infty} + \frac12\|w_z\|_{\infty}\right) \|v\|_{L^2}^2 +\frac12 \|\nablah v\|_{L^2}^2.
\end{align*}
Gronwall's Lemma now implies $v=0$.

Existence: Let $(\Phi_n)_n \subset V$ be 
orthonormal in $H$ with $\partial_{zz}\Phi_n\in L^2(\Omega)$ 
and $\operatorname{span} \{\Phi_n | n\in\N\}$ dense 
in $H$. We set $V_n:=\operatorname{span}\{\Phi_j | 1\leq j \leq n\}$ 
and denote by $P_n: H\to V_n$ the orthogonal projection onto it.

We project the equation onto the finite 
dimensional subspace $V_n$ and we are 
looking for a solution $v_n:[0,T]\to V_n$ 
of the system of ordinary differential equations 
\begin{align}\label{eq:galerkin}
\begin{split}
\langle \partial_t v_n , \Phi_i \rangle_H+\langle  b\cdot\nablah  v_n, \Phi_i \rangle_H+\langle a v_n, \Phi_i \rangle_H\qquad\\
 +\langle w \partial_z v_n, \Phi_i \rangle_H - \langle \Deltah v_n, \Phi_i \rangle_H & = \langle f, \Phi_i \rangle_H, \quad (1\leq i \leq n)  \\
v_n(0) & =P_n v_0.
\end{split}
\end{align}
We have $v_n(t)=\sum_{j=1}^n g_{nj}(t)\Phi_j$ for some $g_{nj}:[0,T]\to \R$ and this yields
 \begin{align*}
\sum_{j=1}^n  \ddt g_{nj}(t)  \langle \Phi_j , \Phi_i \rangle_H+ g_{nj}(t)  \langle  b\cdot\nablah \Phi_j, \Phi_i \rangle_H+ g_{nj}(t)  \langle a\Phi_j, \Phi_i \rangle_H\\
 + g_{nj}(t)  \langle w \partial_z\Phi_j, \Phi_i \rangle_H - g_{nj}(t)  \langle \Deltah \Phi_j, \Phi_i \rangle_H & = \langle f, \Phi_i \rangle_H,  \\
\sum_{j=1}^n g_{nj}(0)\Phi_j & =P_n v_0
\end{align*}
for $1\leq i\leq n$. Denoting 
 \begin{align*}
g_n(t)& :=(g_{1n}(t),...,g_{nn}(t)), \qquad\qquad\quad
f_n(t):=(\langle f(t), \Phi_1 \rangle_H,...,\langle f(t), \Phi_n \rangle_H),\\
a_n(t)&:=(\langle a(t) \Phi_j, \Phi_i \rangle_H)_{1\leq i,j\leq n}, \qquad\quad
b_n(t):=(\langle (b(t)\cdot \nablah) \Phi_j, \Phi_i \rangle_H)_{1\leq i,j\leq n}, \\
w_n(t)&:=(\langle w(t) \partial_z \Phi_j, \Phi_i \rangle_H)_{1\leq i,j\leq n}, \qquad
D_n:=(\langle \nablah  \Phi_j, \nablah  \Phi_i \rangle_H)_{1\leq i,j\leq n}
\end{align*}
this system has the form
\begin{align*}
 \ddt g_n(t)+[a_n(t)+b_n(t) + w_n(t)+ D_n]g_n(t)  =f_n(t),  \\
\sum_{j=1}^n g_{nj}(0)\Phi_j=P_n v_0.
\end{align*}
By standard theory for ordinary differential 
equations there is a solution $g_n\in H^1((0,T),\R)$. 
Multiplying (\ref{eq:galerkin}) by $g_{ni}$ and 
summing over $i$ yields
\begin{multline*}
\langle \partial_t v_n , v_n \rangle_H +\langle a v_n, v_n \rangle_H +\langle  b\cdot\nablah  v_n, v_n \rangle_H
 +\langle w \partial_z v_n, v_n \rangle_H - \langle \Deltah v_n, v_n \rangle_H  \\ = \langle f, v_n \rangle_H.
\end{multline*}
From the periodic boundary conditions we obtain 
\begin{align*}
-\langle \Deltah v_n, v_n \rangle_H=\|\partial_x v_n\|_H^2+\|\partial_y v_n\|_H^2,
\end{align*}
and $w(z=\pm h)=0$ gives 
\begin{align*}
|\langle w \partial_z v_n, v_n \rangle_H | 
& =\left|-\frac12 \langle w_z \cdot v_n, v_n \rangle_{\Omega} +\frac12 \langle w_z \cdot \partial_z v_n, \partial_z  v_n \rangle_{\Omega} \right|
 \leq  \frac12 \|w_z\|_{\infty}  \|v_n\|_H^2.
\end{align*}
With
\begin{align*}
|\langle  b\cdot\nablah  v_n, v_n \rangle_H| & =|\langle  b\cdot\nablah  v_n, v_n \rangle_{\Omega} +\langle  b\cdot\nablah  \partial_z v_n+ (b_z \cdot \nablah) v_n,\partial_z  v_n \rangle_{\Omega} |\\
& \leq \|b\|_{\infty}\|\nablah v_n\|_{L^2}\|v_n\|_{L^2}+ \|b\|_{\infty} \cdot \|\nablah \partial_z v_n\|_{L^2} \|\partial_z v_n\|_{L^2}\\
&\quad+\|b_z\|_{\infty}\|\nablah v_n\|_{L^2} \|\partial_z v_n\|_{L^2}\\
& \leq (\|b\|_{\infty}^2+\|b_z\|_{\infty}^2 )\|v_n\|_H^2 +\frac12 \|\nablah v_n\|_H^2
\end{align*}
and
\begin{align*}
|\langle a v_n, v_n \rangle_H| 
& \leq \|a\|_{\infty} \cdot \|v_n\|_{L^2}^2+ \|a\|_{\infty} \cdot \|\partial_z v_n\|_{L^2}^2+\|a_z\|_{\infty} \cdot \|v_n\|_{L^2}\|\partial_z v_n\|_{L^2}\\
& = \|a\|_{\infty}\cdot \|v_n\|_H^2 +\frac12 \|a_z\|_{\infty} \| v_n\|_H^2
\end{align*}
we get
 \begin{align*}
& \ddt \|v_n\|_H^2 +2\|\nablah v_n\|_H^2 \\
& \qquad\qquad =(2\|a\|_{\infty}+\|a_z\|_{\infty}) \cdot \|v_n\|_H^2+2(\|b\|_{\infty}^2+\|b_z\|_{\infty}^2 ) \cdot \|v_n\|_H^2\\
& \qquad\qquad\quad  + \|w_z\|_{\infty} \cdot \|v_n\|_H^2+ \|\nablah v_n\|_H^2 + \frac12 \|v_n\|_H^2+\frac12 \|\nablah v_n\|_H^2  + 2\|f\|_{V'}^2,
\end{align*}
and thus
 \begin{multline*}
\ddt \|v_n\|_H^2 +\frac12 \|\nablah v_n\|_H^2 \\
\leq \left(\frac12+\|w_z\|_{\infty}+2\|(a,a_z)\|_{\infty}+2\|(b,b_z)\|_{\infty}^2 \right) \cdot \|v_n\|_H^2+ 2\|f\|_{V'}^2.
\end{multline*}
Integration in time
 \begin{multline*}
\|v_n(t)\|_H^2 +\int_0^t\|\nablah v_n(r)\|_H^2 \dr   \leq \|v_0\|_{H}^2+\int_0^t \|f(r)\|_{V'}^2 \dr \\
+ \left(\frac12+\|w_z\|_{\infty}+2\|(a,a_z)\|_{\infty}+2\|(b,b_z)\|_{\infty}^2\right) \cdot \int_0^t \|v_n(r)\|_H^2 \dr
\end{multline*}
and Gronwall's inequality give
 \begin{multline*}
  \|v_n(t)\|_H^2+\int_0^t\|\nablah v_n(r)\|_H^2 \dr \\
  \leq \left(\|v_0\|_{H}^2+2\int_0^t\|f(r)\|_{V'}^2\dr \right) e^{\left(\frac12+\|w_z\|_{\infty}+2\|(a,a_z)\|_{\infty}+2\|(b,b_z)\|_{\infty}^2\right) T}.
\end{multline*}
This estimate yields the boundedness of the 
sequence $(v_n)_n$ in $C^0([0,T],H)\cap L^2((0,T),V)$. 
Analogously the strong convergence in 
$C^0([0,T],H)\cap L^2((0,T),V)$ follows by 
performing the above estimates for $v_m-v_n$.

Let now $\varphi\in C_c^{\infty}([0,T],V)$ be 
of the form $\varphi(t)=\sum_{i=1}^k h_i(t)\Phi_i$ 
with $k\in \N$ and $h_i\in C_c^{\infty}([0,T],\R)$ 
($1\leq i \leq k$). From (\ref{eq:galerkin}) we get with
\begin{align*}
\langle w \partial_z v_n, \varphi \rangle_H 
 = \langle w \partial_z v_n, \varphi \rangle_{\Omega}-\langle w \partial_z  v_n,\partial_{zz} \varphi \rangle_{\Omega} 
 = \langle w \partial_z v_n, (1-\partial_{zz}) \varphi \rangle_{\Omega}
\end{align*}
that
\begin{multline*}
\int_0^T -\langle  v_n , \partial_t \varphi \rangle_H +\langle a v_n, \varphi \rangle_H +\langle  b\cdot\nablah  v_n, \varphi \rangle_H\\
+\langle w \partial_z v_n, (1-\partial_{zz}) \varphi \rangle_{\Omega}  + \langle \nablah  v_n, \nablah  \varphi \rangle_H \dt = \int_0^T  \langle f, \varphi \rangle_H \dt
\end{multline*}
and passing to the limit gives
\begin{multline*}
\int_0^T  -\langle  v , \partial_t \varphi \rangle_H +\langle a v, \varphi \rangle_H +\langle  b\cdot\nablah  v, \varphi \rangle_H\\
+\langle w \partial_z v, (1-\partial_{zz}) \varphi \rangle_{\Omega} + \langle \nablah v, \nablah  \varphi \rangle_H \dt =\int_0^T  \langle f, \varphi \rangle_H\dt.
\end{multline*}
To show that $v$ is a weak solution we 
need to have this equality with scalar 
products in $L^2(\Omega)$ instead of $H$, 
and without the term $(1-\partial_{zz})\varphi$. 
By density arguments the above equality holds for 
all $\varphi\in C_c^{\infty}([0,T],V)$ with 
$\partial_{zz}\varphi\in C_c^{\infty}([0,T],L^2(\Omega))$, 
and taking such a $\varphi$ with 
$\partial_{z}\varphi\in C_c^{\infty}([0,T],V) $ 
and $\partial_{z}\varphi(z=\pm h)=0$ we obtain
\begin{multline*}
-  \langle v , \partial_t (1-\partial_{zz})\varphi \rangle_{T} +\langle a v, (1-\partial_{zz})\varphi \rangle_{T} +\langle  b\cdot\nablah  v, (1-\partial_{zz})\varphi \rangle_{T} \\
+\langle w \partial_z v, (1-\partial_{zz}) \varphi \rangle_{T} + \langle \nablah  v, \nablah  (1-\partial_{zz})\varphi \rangle_{T}  = \langle f, (1-\partial_{zz})\varphi \rangle_{T}.
\end{multline*}
The set
\begin{align*}
\{(1-\partial_{zz})\varphi|\varphi,\partial_{z}\varphi\in C_c^{\infty}([0,T],V),\partial_{z}\varphi(z=\pm h)=0\}
\end{align*}
is dense in $C_c^{\infty}((0,T),L^2((-h,h),H^1_{\per}(G)))$, and so we have
\begin{multline*}
  \langle  v , \partial_t \varphi \rangle_{T}+\langle a v, \varphi \rangle_{T} +\langle  b\cdot\nablah  v, \varphi \rangle_{T}
+\langle w \partial_z v, \varphi \rangle_{T} + \langle \nablah  v, \nablah  \varphi \rangle_{T}  =  \langle f, \varphi \rangle_{T} 
\end{multline*}
for all $\varphi \in C_c^{\infty}((0,T),L^2((-h,h),H^1_{\per}(G)))$.\\
Thus, $v\in H^1((0,T),L^2((-h,h),H^{-1}(G))) \cap C^0([0,T],H)\cap L^2((0,T),V)$ is a weak solution.
\end{proof}

A direct consequence is the following result on $C^{\infty}$-data.
\begin{corollary}\label{cor:linreg}
Let $a,b,w,f \in C^{\infty}([0,T],C^{\infty}_{\per}(\Omega))$ with $w(z=\pm h)=0$ and $v_0 \in C^{\infty}_{\per}(\Omega)$. Then we have for the solution $v$ to (\ref{eq:primequlin}) obtained in Theorem \ref{th:linex} that
\begin{align*}
v  \in C^{\infty}([0,T],C^{\infty}_{\per}(\Omega)).
\end{align*}
\end{corollary}

\section{Primitive equations with horizontal viscosity}
In this section we show that the equation 
(\ref{eq:primequhorvisc}) is well-posed. 
First we prove the local in time well-posedness 
part of Theorem \ref{th:introd1}.
\begin{theorem}\label{th:horviscloc}
Let $s\geq 2$. Then for any  $v_0\in H^s_{\per}(\Omega)$
 with $\divh \overline v_0=0$  there exists 
 a time $T>0$ and a unique strong solution 
 $v$ to (\ref{eq:primequhorvisc}) with
\begin{align*}
 v & \in L^{\infty}((0,T), H^{s}_{\per}(\Omega)) \cap C^0([0,T], H^{s-\kappa}_{\per}(\Omega))\\
\partial_x v,\partial_y v & \in L^2((0,T),H^{s}_{\per}(\Omega))
\end{align*}
for all $\kappa\in(0,1)$.
\end{theorem}

\begin{proof} 
Let $(v_{0,n})_n\subset C_{\per}^{\infty}(\Omega)$ 
with $v_{0,n}\to v_0$ in $H^s(\Omega)$ and 
$\norm{v_n}_{H^s}\leq \norm{v}_{H^s}$.

For $n\in\N$ and $\overline v_{n-1}, \tilde v_{n-1}$, 
$w_{n-1} \in C^{\infty}([0,T],C_{\per}^{\infty}(\Omega))$ 
given let $\overline v_{n}=\overline v_{n}(t,x,y)$ and 
$p_n=p_n(t,x,y)$ be the solution to
\begin{align*}
\partial_t\overline v_{n}+\overline v_{n-1}\cdot\nablah\overline v_{n}- \Deltah \overline v_{n}+\nablah  p_n & =-K(\tilde v_{n-1}), \\
\divh \overline  v_{n} & = 0,\\
\overline v_{n}(t=0)& =\frac{1}{2h}\int_{-h}^{h} v_{0,n}(0,x,y,\xi)\dxi,
\end{align*}
where $K$ is defined as in (\ref{eq:primequcoupling}), 
and $\tilde v_{n}=\tilde v_{n}(t,x,y,z)$ be the solution to
\begin{align*}
&\partial_t \tilde v_{n} + (\tilde v_{n-1} + \overline v_{n-1}) \cdot \nablah  \tilde v_{n}+ \tilde v_{n-1} \cdot \nablah  \overline v_{n} + w_{n-1} \partial_z \tilde v_{n} - \Deltah \tilde v_{n}= K(\tilde v_{n-1}),\\
&\tilde v_{n}(t=0)  =v_{0,n}-\overline v_{0,n},
\end{align*}
both equations are complemented by 
periodic boundary conditions in the 
horizontal directions. We define 
\begin{align*}
w_{n}(t,x,y,z) & =-\int_{-h}^{z}\divh \tilde v_{n} (t,x,y,\xi) \dxi+\frac{z+h}{2h}\int_{-h}^{h}\divh \tilde v_{n} (t,x,y,\xi) \dxi.
\end{align*}
Starting with $\overline v_{0}=\tilde v_{0}=w_{0}=0$ 
the sequence is well defined by Corollary \ref{cor:linreg} 
and known results for the $2D$ Stokes equation.

Note, that this set of equations looks similar 
to (\ref{eq:primequhorviscbar})-(\ref{eq:primequcoupling}), 
but $\tilde v_{n}$ is not average free in the vertical 
direction and therefore we need the correction term in 
the equation for $w_n$ to guarantee that $w_{n}(z=\pm h)=0$. 
However, after passing to the limit the resulting function 
$\tilde v$ will be average free and thus the correction term vanishes.

For $\overline v_n$ we obtain the inequality
\begin{multline*}
\frac12\ddt\norm{\overline v_{n}}^2_{H^s}+\norm{\nablah\overline v_{n}}^2_{H^s} \\ 
\leq c\norm{\overline v_{n-1}}_{H^s}\norm{\overline v_{n}}_{H^s}\norm{\nablah \overline v_{n}}_{H^s}  +c\norm{\tilde v_{n-1}}_{H^s}\norm{\nablah \tilde v_{n-1}}_{H^s}\norm{\overline v_{n}}_{H^s} , 
\end{multline*}
where the last term is due to the 
coupling $K(\tilde v)$. By applying 
$\divh$ to the equation for $\overline v_n$ 
we get that 
\begin{align*}
-\Deltah  p_n  & =\divh( K(\tilde v_{n-1})+\overline v_{n-1} \cdot \nablah  \overline v_n),
\end{align*}
and so we have for the pressure
\begin{align}\label{eq:estimatepressure}
\norm{\nablah  p_n}_{H^s} & \leq c\norm{\tilde v_{n-1}}_{H^s} \norm{\nablah \tilde v_{n-1}}_{H^s}+c\norm{\overline v_{n-1}}_{H^s} \norm{\nablah \overline v_{n}}_{H^s}.
\end{align}
Next we show an estimate for $v_{n}:=\overline v_{n}+\tilde v_{n}$. It fulfills the equation
\begin{align*}
\partial_t  v_{n} +  v_{n-1} \cdot \nablah   v_{n}+ w_{n-1} \partial_z v_{n} - \Deltah  v_{n} + \nablah  p_n  & =0 , \\
\partial_z p_n  & =0 , \\
\divh \overline  v_{n} & = 0,\\
v_{n}(t=0)& = v_{0,n}.
\end{align*}
Applying $\nablatoalpha$ and multiplying with $\nablatoalpha  v_n$ we obtain
\begin{multline*}
\frac12 \partial_t\norm{\nablatoalpha v_{n}}_{L^2}^2+\norm{\nablah\nablatoalpha v_{n}}_{L^2}^2\\
=-\langle \nablatoalpha( v_{n-1} \cdot \nablah v_{n}+w_{n-1} \partial_z   v_{n}+\nablah p_n),\nablatoalpha v_{n}\rangle_{\Omega}.
\end{multline*}
For the pressure term we only have 
$\langle \nablatoalpha\nablah p_n,\nablatoalpha v_{n}\rangle_{\Omega}=0$ 
if $\nablatoalpha$ contains a derivative in the $z$ 
direction, because of the correction term in the 
definition of $w_{n-1}$, but with (\ref{eq:estimatepressure}) 
and $\tilde v_{n}=v_{n}-\overline v_{n}$ it follows 
for $\nablatoalpha=(\partial_x,\partial_y)^{\alpha}$ that
\begin{align*}
& |\langle \nablatoalpha\nablah p_n,\nablatoalpha v_{n}\rangle_{\Omega}|\\
& \leq c(\norm{v_{n-1}}_{H^s} \norm{\nablah v_{n-1}}_{H^s}+\norm{ v_{n-1}}_{H^s} \norm{\nablah \overline v_{n-1}}_{H^s}+\norm{\overline v_{n-1}}_{H^s} \norm{\nablah v_{n-1}}_{H^s}\\
&\qquad +\norm{\overline v_{n-1}}_{H^s} \norm{\nablah \overline v_{n-1}}_{H^s}
+\norm{\overline v_{n-1}}_{H^s} \norm{\nablah \overline v_{n}}_{H^s} )\norm{\nablatoalpha v_{n}}_{L^2}.
\end{align*}
The first part of the nonlinearity can be estimated directly
\begin{align*}
|\langle  \nablatoalpha( v_{n-1} \cdot \nablah v_{n}),\nablatoalpha v_{n} \rangle_{\Omega}|\leq  c\norm{v_{n-1}}_{H^s} \norm{\nablah v_{n}}_{H^s}\norm{\nablatoalpha v_{n}}_{L^2}.
\end{align*}
The other one we write as
\begin{align*}
\langle  \nablatoalpha( w_{n-1} \partial_z v_{n}),\nablatoalpha v_{n} \rangle_{\Omega}= 
&\langle  (\nablatoalpha w_{n-1}) \partial_z  v_{n},\nablatoalpha v_{n} \rangle_{\Omega}+\langle w_{n-1} \nablatoalpha \partial_z v_{n},\nablatoalpha v_{n} \rangle_{\Omega}\\
&+\sum_{0<\alpha'<\alpha}\langle  \nablatoalphastr w_{n-1} \nablatoalphaminstr\partial_z v_{n}),\nablatoalpha v_{n} \rangle_{\Omega}.
\end{align*}
By $w_{n-1}(z=\pm h)=0$ and because of $\norm{\partial_z w_{n-1}}_{L^{\infty}}\leq \norm{\divh v_n}_{H^2}$ we obtain
\begin{align*}
|\langle w_{n-1} \nablatoalpha \partial_z v_{n},\nablatoalpha v_{n} \rangle_{\Omega}| & 
 \leq c \norm{\divh v_{n-1}}_{H^2} \norm{\nablatoalpha v_{n}}_{L^2}^2.
\end{align*}
Let us recall that by Lemma \ref{lemma:lowreg}
\begin{align*}
|\langle fg,h\rangle_{\Omega}|\leq c \norm{f}_{L^2}  \norm{g}_{H^1} \norm{\nablah h}_{L^2}^{1/2}\norm{h}_{L^2}^{1/2}+c \norm{f}_{L^2}  \norm{g}_{H^1} \norm{h}_{L^2}
\end{align*}
holds. This implies 
\begin{align*}
|\langle(\nablatoalpha w_{n-1}) \partial_z  v_{n},\nablatoalpha v_{n}  \rangle_{\Omega}| \leq 
& c \norm{\nablatoalpha w_{n-1}}_{L^2}  \norm{ \partial_z  v_{n}}_{H^1} \norm{\nablah \nablatoalpha v_{n}}_{L^2}^{1/2}\norm{\nablatoalpha v_{n}}_{L^2}^{1/2}\\
& + c\norm{\nablatoalpha w_{n-1}}_{L^2}  \norm{ \partial_z  v_{n}}_{H^1} \norm{\nablatoalpha v_{n}}_{L^2}\\
\leq & c \norm{\divh v_{n-1}}_{H^s} \norm{v_{n}}_{H^2} \norm{\nablah v_{n}}_{H^s}^{1/2}\norm{\nablatoalpha v_{n}}_{L^2}^{1/2}\\
&+ c \norm{\divh v_{n-1}}_{H^s} \norm{v_{n}}_{H^2}\norm{\nablatoalpha v_{n}}_{L^2}.
\end{align*}
Here we have two terms which contain 
third order derivatives, but they come 
with a power strictly less then two, so 
they also can be absorbed by the horizontal 
Laplacian in the end. For the sum we proceed 
similarly and get
\begin{align*}
& \sum_{0<\alpha'<\alpha}|\langle \nablatoalphastr w_{n-1} \nablatoalphaminstr\partial_z v_{n}),\nablatoalpha v_{n} \rangle_{\Omega}| \\
& \qquad\qquad \leq c \sum_{0<\alpha'<\alpha} \norm{\nablatoalphastr w_{n-1}}_{H^1}  \norm{\nablatoalphaminstr \partial_z  v_{n}}_{L^2} \norm{\nablah \nablatoalpha v_{n}}_{L^2}^{1/2}\norm{\nablatoalpha v_{n}}_{L^2}^{1/2}\\
&\qquad\qquad\qquad\qquad\quad +\norm{\nablatoalphastr w_{n-1}}_{H^1}  \norm{\nablatoalphaminstr \partial_z  v_{n}}_{L^2} \norm{\nablatoalpha v_{n}}_{L^2}\\
& \qquad\qquad \leq c \norm{\nablah v_{n-1}}_{H^s} \norm{v_{n}}_{H^s} \norm{\nablah v_{n}}_{H^s}^{1/2}\norm{\nablatoalpha v_{n}}_{L^2}^{1/2}\\
& \qquad\qquad\quad + c \norm{\divh v_{n-1}}_{H^s} \norm{v_{n}}_{H^2}\norm{\nablatoalpha v_{n}}_{L^2}.
\end{align*}
With Young's inequality it follows that
\begin{align*}
&\partial_t\norm{v_{n}}^2_{H^s}+\norm{\nablah v_{n}}_{H^s}^2\\
&\quad \leq  c\left(\norm{v_{n-1}}_{H^s}+\norm{\overline v_{n-1}}^2_{H^s}+\norm{v_{n-1}}_{H^s}^4+\norm{\nablah v_{n-1}}+\norm{\nablah v_{n-1}}^{4/3}_{H^s}\right) \norm{v_{n}}^2_{H^s}\\
&\qquad+\norm{\nablah \overline v_{n}}^2_{H^s}+\frac12 \norm{\nablah \overline v_{n-1}}^2_{H^s}+\frac12\norm{\nablah v_{n-1}}^2_{H^s}.
\end{align*}
Combined with the estimate for $\overline v$ we have 
\begin{align*}
&\partial_t (\norm{v_{n}}^2_{H^s}+\norm{\overline v_{n}}^2_{H^s})+\norm{\nablah v_{n}}_{H^s}^2+\norm{\nablah \overline v_{n}}_{H^s}^2\\
&\quad \leq  c\left(\norm{v_{n-1}}_{H^s}+\norm{\overline v_{n-1}}^2_{H^s}+\norm{v_{n-1}}_{H^s}^4+\norm{\nablah v_{n-1}}+\norm{\nablah v_{n-1}}^{4/3}_{H^s}\right) \norm{v_{n}}^2_{H^2}\\
&\qquad + c\left(\norm{v_{n-1}}^2_{H^s}+\norm{\overline v_{n-1}}^2_{H^s}\right) \norm{\overline v_{n}}^2_{H^s} +\frac12\norm{\nablah\overline v_{n-1}}^2_{H^s}+\frac12\norm{\nablah v_{n-1}}^2_{H^s},
\end{align*}
and Gronwall's inequality yields
\begin{align*}
\begin{split}
&\norm{v_{n}}^2_{H^s}+\norm{\overline v_{n}}^2_{H^s}+\int_0^t \norm{\nablah v_{n}(r)}_{H^s}^2+\norm{\nablah \overline v_{n}(r)}_{H^s}^2 \dr\\
& \leq \left( \norm{v(0)}^2_{H^s}+\norm{\overline v(0)}^2_{H^s}+\frac12\int_0^t \norm{\nablah v_{n-1}(r)}_{H^s}^2+\norm{\nablah \overline v_{n-1}(r)}_{H^s}^2 \dr \right) e^{f_{n-1}(t)},
\end{split}
\end{align*}
where
\begin{align*}
 f_{n-1}(t)= & ct \left( \norm{v_{n-1}}_{L^{\infty}((0,T),H^s)}+\norm{\overline v_{n-1}}^2_{L^{\infty}((0,T),H^s)}+ \norm{v_{n-1}}_{L^{\infty}((0,T),H^s)}^{4}\right)\\
 &+c\int_0^t \norm{\nablah v_{n-1}(r)}_{H^s}+\norm{\nablah v_{n-1}(r)}^{4/3}_{H^s}\dr \\
 \leq &ct \left( \norm{v_{n-1}}_{L^{\infty}((0,T),H^s)}+\norm{\overline v_{n-1}}^2_{L^{\infty}((0,T),H^s)}+ \norm{v_{n-1}}_{L^{\infty}((0,T),H^s)}^{4}\right)\\
 & +c t^{1/3}\norm{\nablah v_{n-1}}_{L^2((0,T),H^s)}^{4/3}+ct^{1/2}\norm{\nablah v_{n-1}}_{L^2((0,T),H^s)}.
\end{align*}

For $T$ sufficiently small this implies that $\norm{v_{n}}_{L^{\infty}((0,T),H^s)}$, $\norm{\overline v_{n}}_{L^{\infty}((0,T),H^s)}$, $\norm{\nablah \! v_{n}}_{L^2((0,T),H^s)}$ and $\norm{\nablah \overline v_{n}}_{L^2((0,T),H^s)}$ are uniformly bounded, and thus also $\tilde v_{n}$ and $\nablah p_n$ are bounded in these norms.

With similar estimates for $v_n-v_m$ we see that $(v_n)_n$ is a Cauchy sequence in
$L^{\infty}((0,T),H^{s-2})$ and $(\nablah v_n)_n$ in $L^2((0,T),H^{s-2})$ for $T$ sufficiently small. Hence, 
we find a $v$ such that
\begin{align*}
 v_n & \to v \quad \text{in } C^0([0,T],H^{s-2}_{\per}(\Omega)) \text{ and} \\
 \nablah v_n & \to \nablah v \quad \text{in } L^2((0,T),H^{s-2}_{\per}(\Omega)).
\end{align*}
By the energy inequality we obtain 
\begin{align*}
v\in L^{\infty}((0,T),H^s_{\per}(\Omega))&,\quad  \nablah v \in L^2((0,T),H^s_{\per}(\Omega))\mbox{ and }\\
\partial_t v \in L^{\infty}((0,T),H^{s-2}_{\per}(\Omega))&,\quad \partial_t \nablah v \in L^2((0,T),H^{s-2}_{\per}(\Omega)).
\end{align*}
Additionally, we have for $\kappa\in(0,2)$ 
\begin{align*}
\norm{v-v_n}_{H^{s-\kappa}} 	& \leq c (\norm{v}_{H^s}+\norm{v_n}_{H^s})^{1-\kappa/2}\norm{v-v_n}_{H^{s-2}}^{\kappa/2}\to 0 \quad(n\to\infty),
\end{align*}
and thus $v\in C^0([0,T],H^{s-\kappa}_{\per}(\Omega))$. 
By the same arguments $(\overline v_n)_n$ and 
$(\tilde v_n)_n$ converge to some $\overline v$ and $\tilde v$ 
with the same regularities as $v$ and the equation 
for $p_n$ yields the convergence of $(p_n)_n$ to some 
$p \in L^{\infty}((0,T),H^s_{\per}(G)) \cap L^2((0,T),H^{s+1}_{\per}(G))$ 
(which is uniquely determined up to a constant) with
\begin{align*}
\partial_t \overline v + \overline v \cdot \nablah  \overline v - \Deltah \overline v + \nablah  p  & =-K(\tilde v) , \\
\divh \overline  v & = 0,\\
\overline v(t=0)& =\frac{1}{2h}\int_{-h}^{h} v_{0}(0,x,y,\xi)\dxi
\end{align*}
and
\begin{align*}
\partial_t \tilde v + \tilde v \cdot \nablah  \tilde v & + \overline v \cdot \nablah  \tilde v+ \tilde v \cdot \nablah  \overline v + w \partial_z \tilde v - \Deltah \tilde v  = K(\tilde v),\\
 w(t,x,y,z) &=-\int_{-h}^{z}\divh \tilde v (t,x,y,\xi) \dxi+\frac{z+h}{2h}\int_{-h}^{h}\divh \tilde v (t,x,y,\xi) \dxi,\\
 \tilde v(t=0) & =v_{0}-\overline v(t=0).
\end{align*}
What is left is to show that $\tilde v$ is 
average free in the $z$ direction. We set 
$u(t,x,y):=\frac{1}{2h}\int_{-h}^{h} \tilde v (t,x,y,\xi) \dxi$. 
From the above equation it follows that 
\begin{align*}
\partial_t u + \overline v \cdot \nablah u + u \cdot \nablah  \overline v  - \Deltah u & = 0,\\
\tilde u(t=0) & =0.
\end{align*}
Multiplication with $u$ in $L^2(G)$ gives
\begin{align*}
\frac 12 \partial_t \norm{u}_{L^2}^2+\norm{ \nablah u}_{L^2}^2
&\leq c \norm{\overline v}_{L^\infty} \norm{u}_{L^2}\norm{ \nablah u}_{L^2},
\end{align*}
and from this we get $\norm{u(t)}_{L^2}^2=0$. Thus, we eventually have
\begin{align*}
w(t,x,y,z)=-\int_{-h}^{z}\divh \tilde v (t,x,y,\xi) \dxi,
\end{align*}
and so $\overline v$ and $\tilde v$ solve 
(\ref{eq:primequhorviscbar})-(\ref{eq:primequcoupling}) 
with $A=\Deltah$, which implies that $v$ is a solution to  (\ref{eq:primequhorvisc}).

The uniqueness and continuous dependence 
on the data of that solution is a direct 
consequence of the energy inequality shown above.
\end{proof}

It follows immediately from the above proof, 
that if in addition $\partial_z v_0\in H^s(\Omega)$ 
then also the regularity of the solution in the 
vertical directions is increased. We will need 
this when investigating the equations with half horizontal 
viscosity.

\begin{corollary}\label{cor:horvisclocz}
Assume that under the conditions of 
Theorem \ref{th:horviscloc} additionally 
$\partial_z v_0\in H^s_{\per}(\Omega)$. 
Then we have for the solution $v$ to 
(\ref{eq:primequhorvisc}) obtained in 
Theorem \ref{th:horviscloc} that additionally 
\begin{align*}
\partial_z v & \in L^{\infty}((0,T), H^{s}_{\per}(\Omega)) \cap C^0([0,T], H^{s-\kappa}_{\per}(\Omega)),\\ \partial_x\partial_z v,\partial_y\partial_z v & \in L^2((0,T),H^{s}_{\per}(\Omega))
\end{align*}
for some $T>0$ and all $\kappa\in(0,1)$.
\end{corollary}

Under the additional boundary condition 
$\partial_z v(z=\pm h)=0$ Cao, Li and 
Titi showed in \cite{CaoLiTiti2016} that for 
$s=2$ the solution to (\ref{eq:primequhorvisc}) 
exists global in time. They construct the solution 
by approximating the system with only horizontal 
viscosity by the system with full viscosity, showing 
bounds on the solution which are independent of the 
vertical viscosity and then letting this vertical 
tend to $0$. The uniform bounds proved by them can 
be carried over directly to our equation, the 
additional boundary condition (which is preserved 
by the equation, see Proposition \ref{pro:boundcondz}) 
is only needed for estimates on the $\partial_{zz}v$ 
term and therefore we obtain literally the same 
estimates for our problem as they do in the limiting 
case. This implies that the local solutions obtained 
in Theorem \ref{th:horviscloc} can be extended 
globally in time, which yields the global in time 
part of Theorem \ref{th:introd1}.

\begin{theorem}\label{th:horviscglob}
For any  $v_0\in H_{\per}^2(\Omega)$ 
with $\divh \overline v_0=0$ and any $T>0$ 
there exists a unique strong solution $v$  
to (\ref{eq:primequhorvisc}) with
\begin{align*}
 v & \in L^{\infty}((0,T), H_{\per}^{2}(\Omega)) \cap C^0([0,T], H_{\per}^{2-\kappa}(\Omega))\\
\partial_x v,\partial_y v & \in L^2((0,T),H_{\per}^{2}(\Omega))
\end{align*}
for all $\kappa\in(0,1)$.
\end{theorem}

If we additionally assume a homogeneous 
Neumann boundary condition for the initial 
data in the vertical direction, then this 
boundary condition is preserved in time.

\begin{proposition}\label{pro:boundcondz}
Assume that under the conditions of Theorem 
\ref{th:horviscloc} additionally $\partial_z v_0(z= h)=0$. 
Then we have also for the solution $v$ to 
(\ref{eq:primequhorvisc}) that $\partial_z v(z=h)=0$, 
and the same holds for $z=-h$.
\end{proposition}
\begin{proof}
Taking the derivative in the vertical direction 
of (\ref{eq:primequhorvisc}) gives
\begin{align*}
\partial_t \partial_z v 
+ v \cdot \nablah  \partial_z v
+ \partial_z v \cdot \nablah  v 
+ \partial_z w \partial_z v
+ w \partial_{zz} v
- \Deltah \partial_z v   = 0,
\end{align*}
and for $z= h$ we obtain for $u(t,x,y):= \partial_z v(z= h)$
\begin{align*}
\partial_t u
+ v(z= h) \cdot \nablah u
+ u \cdot \nablah  v(z= h) 
-\divh v(z= h)  u
- \Deltah u   & = 0,
\end{align*}
with $u(t=0) =0$. Multiplication with $u$ in $L^2(G)$ gives
\begin{align*}
\frac 12 \partial_t \norm{u}_{L^2}^2+\norm{ \nablah u}_{L^2}^2
&\leq c \norm{v}_{L^\infty} \norm{u}_{L^2}\norm{ \nablah u}_{L^2}
\end{align*}
an this implies $\norm{u(t)}_{L^2}=0$. For $z=-h$ we proceed analogously.
\end{proof}

\section{Half horizontal viscosity}
Here we show that the equation at least 
is locally well-posed, when we only have 
half horizontal viscosity. 
We first turn to the more involved case 
(\ref{eq:primequhalfhorvisc}) 
and prove Theorem \ref{th:introd2}.  
\subsection{Proof of Theorem \ref{th:introd2}} 
Let us briefly describe the strategy of the 
proof. We will assume initial data 
$v_0,\partial_z v_0 \in H^s_{\per}$ 
with $s\geq 3$ for which additionally a Rayleigh condition
\begin{align*}
\frac{1}{\eta} \leq \frac{1}{|\partial_{zz} v_{1}(t=0)|} \leq \eta,\quad \frac{1}{\eta} \leq \frac{1}{|\partial_{zz} v_2(t=0)|} \leq \eta
\end{align*}
holds for some $\eta>1$, i.e. 
$\partial_z v_0 \in H^s_{\per,\eta}(\Omega)$. 
We have $v_{zz}(t=0)\in C^0(\Omega)$ because 
of $s\geq3$, so the above point-wise condition makes sense.

We consider for $\varepsilon>0$ the system 
(\ref{eq:primequhalfhorvisc}) with $A_{\per}$ replaced by 
\begin{align*}
 A_{\varepsilon}=\left(\begin{matrix} \varepsilon\partial_{xx}+\partial_{yy} & 0 \\ 0 & \partial_{xx}+\varepsilon\partial_{yy}\end{matrix}\right).
\end{align*}
Corollary \ref{cor:horvisclocz} guarantees 
the existence of a solution for $\varepsilon>0$ 
and implies the continuity of $\partial_{zz}v$, 
so we have a Rayleigh condition 
$\frac{1}{2\eta} \leq \frac{1}{|\partial_{zz} v_{i}|} \leq 2\eta$ also in some initial time interval.
Using this we show an $\varepsilon$-independent estimate 
for the $H^s$-norm of the corresponding solutions and 
perform the limit $\varepsilon\to 0$.

The main difficulty will be to control 
the highest derivatives in the horizontal directions 
$\norm{(\partial_x,\partial_y)^{\alpha} v(t)}_{L^2(\Omega)}$ 
for $|\alpha|=s$. To obtain an estimate 
for this norms we follow the idea by Masmoudi 
and Wong for the $2D$ primitive Euler 
equation and use the Rayleigh condition 
to obtain bounds on $\norm{\partial_z v(t)}_{H^s(\Omega)}$, 
but in difference to the $2D$ Euler case these 
bounds cannot be carried over to bounds on $v$ directly.

To work around that problem we combine this 
approach with the idea of Cao and Titi to split 
$v$ via $v=\overline v +\tilde v$ and consider the 
set of coupled equations 
(\ref{eq:primequhorviscbar})-(\ref{eq:primequcoupling}) 
for $A=A_{\varepsilon}$. We can deduce bounds for the 
mean value $\overline v$ from the $2D$ Navier-Stokes 
equation (\ref{eq:primequhorviscbar}), and by Poincar\'e's 
inequality it suffices to have a bound for $\partial_z v$ to 
control $\tilde v$. We divide this proof into three 
steps, the estimates on the barotropic mode, the estimates 
on the baroclinic mode and the convergence of the solutions for 
$\varepsilon>0$ to the solution of (\ref{eq:primequhalfhorvisc}) 
when $\varepsilon$ tends to $0$.

\subsubsection*{Estimates for the barotropic mode}
The estimates for $\overline v$ are the easier part, 
only the coupling-term  has to be handled with some care. 
We obtain the following result.
\begin{lemma}\label{lemma:baroclinic}
Let $s\geq3$, $\varepsilon>0$, $v_0 \in H_{\per}^s(\Omega)$ 
with $\divh \overline v_0=0$, $\partial_z v_0 \in H_{\per}^s(\Omega)$ 
and $v=\overline v+ \tilde v$ be the solution to 
(\ref{eq:primequhorviscbar})-(\ref{eq:primequcoupling}) 
for $A=A_{\varepsilon}$ according to Corollary \ref{cor:horvisclocz}. 
Then for any $\delta>0$
\begin{align*}
&\ddt \frac12\norm{\overline{v}}^2_{H^s(G)}+\norm{\partial_y \overline{v}_{1}}_{H^s(G)}^2+\norm{\partial_x \overline{v}_{2}}_{H^s(G)}^2+\varepsilon \norm{\partial_x \overline{v}_{1}}_{H^s(G)}^2+\varepsilon\norm{\partial_y \overline{v}_{2}}_{H^s(G)}^2\\
&\qquad\qquad \leq c \norm{\overline{v}}_{H^{s}(G)}^3+\frac{c}{\delta} \norm{\partial_z \tilde v}^2_{H^{s}(\Omega)}\norm{\overline v}^2_{H^{s}(G)}\\
&\qquad\qquad\quad +\delta( \norm{\partial_y \partial_z\tilde v_1}^2_{H^{s}(\Omega)} + \norm{\partial_y \overline v_1}^2_{H^{s}(G)}+ \norm{\partial_x \overline v_2}^2_{H^{s}(G)}+\norm{\partial_x \partial_z\tilde v_2}^2_{H^{s}(\Omega)})
\end{align*}
holds.
\end{lemma}
\begin{proof}
Applying $\nablatoalpha$ to (\ref{eq:primequhorviscbar}) 
and multiplying it in $L^2(G)$ 
by $\nablatoalpha \overline{v}$ 
for $|\alpha|\leq s$ yields 
\begin{multline*}
\ddt \frac12 \norm{\nablatoalpha \overline{v}}^2_{L^2}+\norm{\nablatoalpha \partial_y \overline{v}_{1}}_{L^2}^2+\norm{\nablatoalpha \partial_x \overline{v}_{2}}_{L^2}^2+\varepsilon \norm{\nablatoalpha \partial_x \overline{v}_{1}}_{L^2}^2+\varepsilon\norm{\nablatoalpha \partial_y \overline{v}_{2}}_{L^2}^2\\
=- \langle \nablatoalpha (\overline{v} \cdot \nablah  \overline{v}), \nabla^{\alpha }\overline{v}\rangle_{G}
- \langle \nablatoalpha \nablah  p, \nabla^{\alpha }\overline{v}\rangle_{G}
-\langle \nablatoalpha K(\tilde v), \nabla^{\alpha }\overline{v}\rangle_{G}.
\end{multline*}
Due to the divergence free condition, 
the periodic boundary conditions and $s\geq 3$ we get
\begin{align*}
 \langle \nablatoalpha \nablah  p, \nabla^{\alpha }\overline{v}\rangle_{G} = 0 \quad \mbox{ and } \quad |\langle \nablatoalpha (\overline{v} \cdot \nablah  \overline{v}), \nabla^{\alpha }\overline{v}\rangle_{G}| & \leq c \norm{\overline{v}}_{H^{s}}^3.
\end{align*}
For the coupling-term we have
\begin{align*}
2h\langle \nablatoalpha K(\tilde v),\! \nabla^{\alpha }\overline{v}\rangle_{G}
& = \int_{-h}^h  \langle \nablatoalpha[\partial_x(\tilde v_1(\cdot,\cdot,\xi))^2+\partial_y (\tilde v_2(\cdot,\cdot,\xi) \tilde v_1(\cdot,\cdot,\xi))], \nabla^{\alpha }\overline{v}_1\rangle_{G}  \dxi\\
&\quad +\! \int_{-h}^h \! \langle \nablatoalpha[\partial_x(\tilde v_1(\cdot,\cdot,\xi) \tilde v_2(\cdot,\cdot,\xi))\! + \partial_y (\tilde v_2(\cdot,\cdot,\xi))^2 ], \nabla^{\alpha }\overline{v}_2\rangle_{G}  \dxi.
\end{align*}
For the first integrand we obtain
\begin{align*}
 \langle \nablatoalpha&[\partial_x(\tilde v_1(\cdot,\cdot,\xi))^2+\partial_y (\tilde v_2(\cdot,\cdot,\xi) \tilde v_1(\cdot,\cdot,\xi))], \nabla^{\alpha }\overline{v}_1\rangle_{G} \\
& =-\langle \nablatoalpha(\tilde v_1(\cdot,\cdot,\xi))^2, \nabla^{\alpha }\partial_x \overline{v}_1\rangle_{G} -\langle \nablatoalpha[\tilde v_2(\cdot,\cdot,\xi) \tilde v_1(\cdot,\cdot,\xi)],  \nabla^{\alpha } \partial_y\overline{v}_1\rangle_{G}\\
& =\langle \nablatoalpha(\tilde v_1(\cdot,\cdot,\xi))^2, \nabla^{\alpha }\partial_y \overline{v}_2\rangle_{G} -\langle \nablatoalpha[\tilde v_2(\cdot,\cdot,\xi) \tilde v_1(\cdot,\cdot,\xi)],  \nabla^{\alpha } \partial_y\overline{v}_1\rangle_{G}\\
& =-\langle \nablatoalpha[2\tilde v_1(\cdot,\cdot,\xi)\partial_y \tilde v_1(\cdot,\cdot,\xi)], \nabla^{\alpha }\overline{v}_2\rangle_{G} -\langle \nablatoalpha[\tilde v_2(\cdot,\cdot,\xi) \tilde v_1(\cdot,\cdot,\xi)],  \nabla^{\alpha } \partial_y\overline{v}_1\rangle_{G},
\end{align*}
and thus
\begin{align*}
& |\langle \nablatoalpha[\partial_x(\tilde v_1(\cdot,\cdot,\xi))^2+\partial_y (\tilde v_2(\cdot,\cdot,\xi) \tilde v_1(\cdot,\cdot,\xi))], \nabla^{\alpha }\overline{v}_1\rangle_{G}| \\
&\qquad\qquad\qquad\qquad \leq c \norm{\tilde v_1(\cdot,\cdot,\xi)}_{H^{s}(G)}\norm{\partial_y \tilde v_1(\cdot,\cdot,\xi)}_{H^{s}(G)}\norm{\overline v_2}_{H^{s}(G)}\\
&\qquad\qquad\qquad\qquad\quad +c \norm{\tilde v_2(\cdot,\cdot,\xi)}_{H^{s}(G)}\norm{\tilde v_1(\cdot,\cdot,\xi)}_{H^{s}(G)}\norm{\partial_y \overline v_1}_{H^{s}(G)}.
\end{align*}
Analogously it follows for the second integrand
\begin{align*}
& |\langle \nablatoalpha[\partial_x(\tilde v_1(\cdot,\cdot,\xi) \tilde v_2(\cdot,\cdot,\xi)) + \partial_y (\tilde v_2(\cdot,\cdot,\xi))^2 ], \nabla^{\alpha }\overline{v}_2\rangle_{G} | \\
&\qquad\qquad\qquad\qquad \leq c  \norm{\tilde v_1(\cdot,\cdot,\xi)}_{H^{s}(G)}\norm{\tilde v_2(\cdot,\cdot,\xi)}_{H^{s}(G)}\norm{\partial_x \overline v_2}_{H^{s}(G)}\\
&\qquad\qquad\qquad\qquad\quad +c \norm{\tilde v_2(\cdot,\cdot,\xi)}_{H^{s}(G)}\norm{\partial_x \tilde v_2(\cdot,\cdot,\xi)}_{H^{s}(G)}\norm{\overline v_1}_{H^{s}(G)}.
\end{align*}
Using
\begin{align*}
&\int_{-h}^{h} \norm{\tilde v_1(\cdot,\cdot,\xi)}_{H^{s}(G)}\norm{\partial_y \tilde v_1(\cdot,\cdot,\xi)}_{H^{s}(G)}\dxi\cdot \norm{\overline v_2}_{H^{s}(G)}\\
&\qquad\qquad\qquad\qquad\leq \norm{\tilde v_1}_{L^2((-h,h),H^{s}(G))}\norm{\partial_y \tilde v_1}_{L^2((-h,h),H^{s}(G))} \norm{\overline v_2}_{H^{s}(G)}\\
&\qquad\qquad\qquad\qquad\leq \norm{\tilde v_1}_{H^{s}(\Omega)}\norm{\partial_y \tilde v_1}_{H^{s}(\Omega)}  \norm{\overline v_2}_{H^{s}(G)}
\end{align*}
we get
\begin{align*}
& |\langle \nablatoalpha K(\tilde v), \nabla^{\alpha }\overline{v}\rangle_{G}| \\
& \qquad\qquad \leq c \norm{\tilde v_1}_{H^{s}(\Omega)} (\norm{\partial_y \tilde v_1}_{H^{s}(\Omega)}  \norm{\overline v_2}_{H^{s}(G)}
+\norm{\tilde v_2}_{H^{s}(\Omega)}  \norm{\partial_y \overline v_1}_{H^{s}(G)})\\
& \qquad\qquad\quad +c \norm{\tilde v_2}_{H^{s}(\Omega)} (\norm{\tilde v_1}_{H^{s}(\Omega)} \norm{\partial_x \overline v_2}_{H^{s}(G)}
+\norm{\partial_x \tilde v_2}_{H^{s}(\Omega)}  \norm{\overline v_1}_{H^{s}(G)})\\
& \qquad\qquad \leq \frac{c}{\delta} \norm{\tilde v}^2_{H^{s}(\Omega)}\norm{\overline v}^2_{H^{s}(G)}\\
& \qquad\qquad\quad +\delta( \norm{\partial_y \tilde v_1}^2_{H^{s}(\Omega)}+ \norm{\partial_y \overline v_1}^2_{H^{s}(G)}+ \norm{\partial_x \overline v_2}^2_{H^{s}(G)}+\norm{\partial_x \tilde v_2}^2_{H^{s}(\Omega)})
\end{align*}
for any $\delta>0$. This leads after summing over 
$\alpha$ and with Poincar\'e's inequality 
for $\tilde v$ the stated estimate.
\end{proof}

\subsubsection*{Estimates for the baroclinic mode}
Here we prove estimates for the vertical 
derivative $\partial_z v$, which is given by the equation
\begin{align}\label{eq:primequvz}
\begin{split}
\partial_t \partial_z v 
+ v \cdot \nablah  \partial_z v
+ \partial_z v \cdot \nablah  v 
+ \partial_z w \partial_z v
+ w \partial_{zz} v
- A_{\varepsilon}\partial_z v   = 0.
\end{split}
\end{align}
The straight forward part is to estimate the 
lower derivatives and those which contain at 
least one derivative in the vertical direction 
(because $w$ and $w_z$ have the same regularity 
with respect to $x$ and $y$).
We multiply the equation by $\nablatoalpha \partial_z v$,
\begin{multline}\label{eq:multplyerbarotropic}
\langle \ddt \nablatoalpha \partial_z v,\nablatoalpha \partial_z v\rangle_{\Omega}-\langle A_{\varepsilon} \nablatoalpha \partial_z v,\nablatoalpha \partial_z v\rangle_{\Omega}\\
=-\langle \nablatoalpha  (w \partial_{zz} v + v \cdot \nablah  \partial_z v-\partial_z v \cdot \nablah  v+ \partial_z v\; \divh v),\nablatoalpha \partial_z v \rangle_{\Omega}.
\end{multline}

The following lemma is a direct consequence of the fact that for $s\geq 3$ the first order derivatives of $v$ and $\partial_z v$ are in $L^{\infty}(\Omega)$.

\begin{lemma}\label{lemma:prepestimatebarotropic1}
Let $s\geq3$, $\varepsilon>0$, $v_0\in H_{\per}^s(\Omega)$ 
with $\divh \overline v_0=0$, $\partial_z v_0 \in H_{\per}^s(\Omega)$ 
and $v=\overline v+\tilde v$ be the solution to 
(\ref{eq:primequhorviscbar})-(\ref{eq:primequcoupling}) 
for $A=A_{\varepsilon}$ according to Corollary 
\ref{cor:horvisclocz}. Then we have for $|\alpha|<s$ 
or $\nabla^{\alpha }=\partial_z \nablatoalphastr$ with 
$|\alpha'|=s-1$ 
\begin{align*}
\frac12 \ddt&\norm{\nablatoalpha \partial_z v}_{L^2}^2  +\norm{\nablatoalpha \partial_z \partial_y v_{1}}_{L^2}^2+\norm{\nablatoalpha \partial_z \partial_x v_{2}}_{L^2}^2\\
&\qquad +\varepsilon \norm{\nablatoalpha \partial_z \partial_x v_{1}}_{L^2}^2+\varepsilon\norm{\nablatoalpha \partial_z \partial_y v_{2}}_{L^2}^2
 \leq c \; (\norm{v}_{H^s} +\norm{\partial_z v}_{H^s}) \norm{\partial_z v}_{H^s}^2.
\end{align*}
\end{lemma}

We replace the multiplier $\nablatoalpha \partial_z v_i$ 
in (\ref{eq:multplyerbarotropic}) by 
$\frac{\nablatoalpha \partial_z v_{i}}{\partial_{zz} v_{i}} $ 
to get an estimate if $\nablatoalpha=(\partial_x,\partial_y)^{\alpha}$ 
with $|\alpha|=s$. In the next lemma we give the estimates 
for the different terms, this is the key step in the proof 
of our local well-posedness result. We write here 
$A_{\varepsilon,1}=\varepsilon \partial_{xx}+\partial_{yy}$ 
and $A_{\varepsilon,2}=\partial_{xx}+\varepsilon \partial_{yy}$.

\begin{lemma}\label{lemma:prepestimatebarotropic2}
Let $s\geq3$, $\eta>1$, $\varepsilon>0$, $v_0\in H_{\per}^s(\Omega)$ 
with $\divh \overline v_0=0$, $ \partial_z v_0 \in H_{\per,\eta}^s(\Omega)$ 
and $v=\overline v+\tilde v$ be the solution to 
(\ref{eq:primequhorviscbar})-(\ref{eq:primequcoupling}) 
for $A=A_{\varepsilon}$ according to Corollary \ref{cor:horvisclocz}. 
Then there exists a time $T$ such that 
$\partial_z v(t)\in H_{\per,2\eta}^s(\Omega)$ ($t\leq T$) 
and for $\nablatoalpha=(\partial_x,\partial_y)^{\alpha}$ 
with $|\alpha|=s$ the following identities and estimates hold:\\
a)
\begin{align*}
\langle \nablatoalpha A_{\varepsilon,1} \partial_z v_{1}
, \frac{\nablatoalpha \partial_z v_{1}}{\partial_{zz} v_{1}} \rangle_{\Omega}
\leq& -\varepsilon\frac{1}{2\eta}\norm{\nablatoalpha \partial_{x} \partial_z v_{1}}^2_{L^2}-\frac{1}{2\eta}\norm{\nablatoalpha \partial_{y} \partial_z v_{1}}^2_{L^2}\\
&+c \eta^2\norm{\partial_{z} v_{1}}_{H^{s}}^{5/4}(\varepsilon\norm{\partial_{x}\partial_{z} v_{1}}_{H^{s}}^{7/4}+ \norm{\partial_{y}\partial_{z} v_{1}}_{H^{s}}^{7/4}) 
\end{align*}
and
\begin{align*}
\langle \nablatoalpha A_{\varepsilon,2} \partial_z v_{2}
, \frac{\nabla^{\alpha }\partial_z v_{2}}{\partial_{zz} v_{2}} \rangle_{\Omega}
\leq& - \frac{1}{2\eta}\norm{\nablatoalpha \partial_{x} \partial_z v_{2}}^2_{L^2}-\varepsilon\frac{1}{2\eta}\norm{\nablatoalpha \partial_{y}\partial_z v_{2}}^2_{L^2}\\
&+c \eta^2\norm{\partial_{z} v_{1}}_{H^{s}}^{5/4} (\norm{\partial_{x}\partial_{z} v_{2}}_{H^{s}}^{7/4}+ \varepsilon \norm{\partial_{y}\partial_{z} v_{2}}_{H^{s}}^{7/4}).
\end{align*}
b) 
\begin{align*}
\langle \nablatoalpha (v \cdot \nablah  \partial_z v_{i}), \frac{\nablatoalpha \partial_z v_{i}}{\partial_{zz} v_{i}} \rangle_{\Omega}
 = & \frac{1}{2}\langle \frac{  v\cdot\nablah \partial_{zz} v_{i}}{(\partial_{zz} v_{i})^2}- \frac{\divh v }{\partial_{zz} v_{i}}, (\nablatoalpha \partial_z v_{i})^2\rangle_{\Omega}\\
&+\sum_{\alpha'<\alpha}\langle \nablatoalphaminstr v \cdot \nablah  \nablatoalphastr \partial_z v_{i}, \frac{\nablatoalpha \partial_z v_{i}}{\partial_{zz} v_{i}} \rangle_{\Omega}
\end{align*}
with
\begin{align*}
\left|\sum_{\alpha'<\alpha}\langle \nablatoalphaminstr v \cdot \nablah  \nablatoalphastr \partial_z v_{i}
, \frac{\nablatoalpha \partial_z v_{i}}{\partial_{zz} v_{i}} \rangle_{\Omega}\right|\leq c\eta \norm{v}_{H^s}\norm{\partial_z v}_{H^s}\norm{\partial_z v_{i}}_{H^s}.
\end{align*}
c)
\begin{align*}
\langle \nablatoalpha (w \partial_{zz} v_{i}) &,\frac{\nablatoalpha \partial_z v_{i}}{\partial_{zz} v_{i}} \rangle_{\Omega}\\
= & \frac{1}{2} \langle \frac{ w \partial_{zzz} v_{i}}{(\partial_{zz} v_{i})^2}+ \frac{\divh v }{\partial_{zz} v_{i}},(\nablatoalpha \partial_z v_{i})^2\rangle_{\Omega}+\langle \nablatoalpha (\partial_x v_{1}+\partial_y v_{2}),\nablatoalpha v_{i}\rangle_{\Omega}\\
&+\sum_{0<\alpha'<\alpha}\langle \nablatoalphaminstr w  \nablatoalphastr \partial_{zz} v_{i})
, \frac{\nablatoalpha\partial_z v_{i}}{\partial_{zz} v_{i}} \rangle_{\Omega}
\end{align*}
with
\begin{align*}
\left|\sum_{0<\alpha'<\alpha}\langle \nablatoalphaminstr w  \nablatoalphastr \partial_{zz} v_{i}
, \frac{\nablatoalpha\partial_z v_{i}}{\partial_{zz} v_{i}} \rangle_{\Omega}\right|\leq c\eta\norm{v}_{H^s}\norm{\partial_z v}_{H^s} \norm{\partial_z v_{i}}_{H^s}.
\end{align*}
d)
\begin{align*}
&\left|\langle \nablatoalpha  (\partial_z v \cdot \nablah  v_{1}+\partial_z w \partial_z v_{1}), \frac{\nablatoalpha\partial_z v_{1}}{\partial_{zz} v_{1}} \rangle_{\Omega}\right|\leq c\eta \norm{v}_{H^s}\norm{\partial_z v}_{H^s}^2\\
&\qquad + c\eta\norm{\partial_z v}_{H^s}( \norm{v}_{H^s}+\norm{\partial_z v}_{H^s})\norm{\partial_y v_{1}}_{H^s} + c\eta^2\norm{v}_{H^s}\norm{\partial_z v}^2_{H^s}\norm{\partial_{z}\partial_y v_{1}}_{H^3}
\end{align*}
and
\begin{align*}
&\left|\langle \nablatoalpha (\partial_z v \cdot \nablah  v_{2}+\partial_z w \partial_z v_{2}), \frac{\nablatoalpha\partial_z v_{2}}{\partial_{zz} v_{2}} \rangle_{\Omega}\right|\leq c\eta \norm{v}_{H^s}\norm{\partial_z v}_{H^s}^2\\
&\qquad + c\eta\norm{\partial_z v}_{H^s}( \norm{v}_{H^s}+\norm{\partial_z v}_{H^s})\norm{\partial_x v_{2}}_{H^s} + c\eta^2\norm{v}_{H^s}\norm{\partial_z v}^2_{H^s}\norm{\partial_{z}\partial_x v_{2}}_{H^3}.
\end{align*}
e) For the time derivative we get
\begin{align*}
&\langle (\nablatoalpha \ddt \partial_z v_{i},\frac{\nablatoalpha \partial_z v_{i}}{\partial_{zz} v_{i}} \rangle _{\Omega}=\frac{1}{2} \ddt \norm{\frac{\nablatoalpha \partial_z v_{i}}{\sqrt{\partial_{zz} v_{i}}}}_{L^2}^2\\
& \qquad -\frac{1}{2} \langle (\nablatoalpha \partial_z v_{i})^2, \frac{v \cdot \partial_{zz} \nablah  v_{i}
+ w \partial_{zzz} v_{i}}{(\partial_{zz} v_{i})^2} \rangle_{\Omega}+\frac{1}{2} \langle (\nablatoalpha \partial_z v_{i})^2, \frac{A_{\varepsilon,i}\partial_{zz} v_{i}}{(\partial_{zz} v_{i})^2} \rangle_{\Omega}\\
& \qquad -\frac{1}{2} \langle (\nablatoalpha \partial_z v_{i})^2, \frac{\partial_{zz} v \cdot \nablah  v_{i}
+ \partial_{zz} w \partial_{z} v_{i}
+ 2\partial_{z} v \cdot \nablah \partial_z v_{i} 
+ 2 \partial_z w \partial_{zz} v_{i}}{(\partial_{zz} v_{i})^2} \rangle_{\Omega}
\end{align*}
with
\begin{align*}
|\langle (\nablatoalpha \partial_z v_{1})^2,  \frac{A_{\varepsilon,1}\partial_{zz} v_1}{(\partial_{zz} v_{1})^2} \rangle_{\Omega}| 
\leq & c \eta^2 \norm{\partial_{z}v_1}_{H^s}^{5/4}  (\varepsilon \norm{\partial_{x}\partial_{z}v_1}_{H^s}^{7/4}+  \norm{\partial_{y}\partial_{z}v_1}_{H^s}^{7/4})\\
& +c\eta^3 \norm{\partial_{z} v_1}^{3/2}_{H^s}(\varepsilon\norm{\partial_{x}\partial_{z} v_1}^{3/2}_{H^3}+\norm{\partial_{y}\partial_{z} v_1}^{3/2}_{H^3}),
\end{align*}
\begin{align*}
|\langle (\nablatoalpha \partial_z v_{2})^2,  \frac{A_{\varepsilon,2}\partial_{zz} v_2}{(\partial_{zz} v_{2})^2} \rangle_{\Omega}| 
\leq & c \eta^2 \norm{\partial_{z}v_2}_{H^s}^{5/4}  (\norm{\partial_{x}\partial_{z}v_2}_{H^s}^{7/4}+  \varepsilon \norm{\partial_{y}\partial_{z}v_2}_{H^s}^{7/4})\\
& +c\eta^3 \norm{\partial_{z} v_2}^{3/2}_{H^s}(\norm{\partial_{x}\partial_{z} v_2}^{3/2}_{H^3}+\varepsilon \norm{\partial_{y}\partial_{z} v_2}^{3/2}_{H^3})
\end{align*}
and
\begin{multline*}
\norm{ \frac{\partial_{zz} v \cdot \nablah  v_{i}
+ \partial_{zz} w \partial_{z} v_{i}
+  2\partial_{z} v \cdot \nablah \partial_z v_{i} 
+ 2 \partial_z w \partial_{zz} v_{i}}{(\partial_{zz} v_{i})^2}}_{L^{\infty}}\\
\leq c\eta^2 (\norm{\partial_{z} v}_{H^3}^2+\norm{\partial_{z} v}_{H^3}\norm{v}_{H^3}).
\end{multline*}
\end{lemma}

\begin{proof}
We have $\partial_z v(t=0)\in H_{\per,\eta}^s(\Omega)$ 
and $\partial_z v$ is continuous, so there exists a $T>0$ 
such that $\partial_z v(t)\in H^s_{\per,2\eta}(\Omega)$ 
for $t\leq T$. We assume in the following for simplicity 
$\partial_{zz} v_{i}>0$.\\ 
a) By integration by parts we obtain
\begin{align*} 
\langle \nablatoalpha \partial_{xx}  \partial_z v_{i}
, & \frac{\nabla^{\alpha }\partial_z v_{i}}{\partial_{zz} v_{i}} \rangle_{\Omega}\\
& =-\norm{\frac{\nablatoalpha\partial_{x} \partial_z v_{i}}{\sqrt{\partial_{zz} v_{i}}}}^2_{L^2}+\langle \nablatoalpha \partial_{x}  \partial_z v_{i}, \frac{\partial_{x}\partial_{zz} v_{i}}{(\partial_{zz} v_{i})^2} \nablatoalpha \partial_z v_{i} \rangle_{\Omega}\\
& \leq -\frac{1}{2\eta}\norm{\nablatoalpha\partial_{x} \partial_z v_{i}}^2_{L^2}+ \eta^2\norm{\partial_{x}\partial_{zz} v_{i}}_{L^{\infty}} \norm{\nablatoalpha \partial_x \partial_z v_{i}}_{L^2} \norm{\nablatoalpha \partial_z v_{i}}_{L^2}\\
& \leq -\frac{1}{2\eta}\norm{\nablatoalpha\partial_{x} \partial_z v_{i}}^2_{L^2}+ c\eta^2\norm{\partial_{x}\partial_{z} v_{i}}_{H^{s}}^{7/4} \norm{\partial_{z} v_{i}}_{H^{1}}^{1/4}  \norm{\nablatoalpha \partial_z v_{i}}_{L^2}
\end{align*}
where we used Lemma \ref{lemma:lowreg} to estimate 
$\norm{\partial_{x}\partial_{zz} v_{i}}_{L^{\infty}}$. 
In the same way we get this estimate for the $y$-derivatives 
and adding them up for $i=1,2$ we obtain the assertion.\\
b)
For 
\begin{align*}
\langle \nablatoalpha (v \cdot \nablah  \partial_z v_{i}), \frac{\nablatoalpha \partial_z v_{i}}{\partial_{zz} v_{i}} \rangle_{\Omega}
 =& \langle v \cdot \nablah  \nablatoalpha \partial_z v_{i})
, \frac{\nablatoalpha \partial_z v_{i}}{\partial_{zz} v_{i}} \rangle_{\Omega}\\
 & +\sum_{\alpha'<\alpha}\langle \nablatoalphaminstr v \cdot \nablah  \nablatoalphastr \partial_z v_{i})
, \frac{\nablatoalpha \partial_z v_{i}}{\partial_{zz} v_{i}} \rangle_{\Omega}
\end{align*}
we have
\begin{align*}
\langle v \cdot \nablah  \nablatoalpha \partial_z v_{i})
, \frac{\nablatoalpha\partial_z v_{i}}{\partial_{zz} v_{i}} \rangle_{\Omega}
& =-\frac{1}{2}\langle \frac{\partial_x v_{1}}{\partial_{zz} v_{i}}-\frac{v_{1}\partial_x\partial_{zz} v_{i} }{(\partial_{zz} v_{i})^2},(\nablatoalpha \partial_z v_{i})^2 \rangle_{\Omega}\\
& \quad -\frac{1}{2}\langle \frac{\partial_y v_{2}}{\partial_{zz} v_{i}}-\frac{v_{2}\partial_y\partial_{zz} v_{i} }{(\partial_{zz} v_{i})^2},(\nablatoalpha \partial_z v_{i})^2 \rangle_{\Omega}\\
&= -\frac{1}{2}\langle \frac{\divh v }{\partial_{zz} v_{i}}- \frac{  v\cdot\nablah \partial_{zz} v_{i}}{(\partial_{zz} v_{i})^2}, (\nablatoalpha \partial_z v_{i})^2\rangle_{\Omega}
\end{align*}
and all the terms in the sum contain only 
derivatives of order less or equal $s$, so
\begin{align*}
|\langle  \nablatoalphaminstr v \cdot \nablah  \nablatoalphastr \partial_z v_{i})
, \frac{\nablatoalpha \partial_z v_{i}}{\partial_{zz} v_{i}} \rangle_{\Omega}|\leq c\eta \norm{v}_{H^s}\norm{\partial_z v}_{H^s}\norm{\nablatoalpha\partial_z v_{i}}_{L^2}.
\end{align*}
c) This is shown similarly to b), 
but due to the bad regularity of $w$ 
we have to separate two terms from the sum, 
\begin{align*}
\langle \nablatoalpha (w \partial_{zz} v_{i}),\frac{\nablatoalpha \partial_z v_{i}}{\partial_{zz} v_{i}} \rangle_{\Omega}
=&\langle \nablatoalpha w \partial_{zz} v_{i},\frac{\nablatoalpha \partial_z v_{i}}{\partial_{zz} v_{i}} \rangle_{\Omega}
+\langle w \partial_{zz}\nablatoalpha v_{i},\frac{\nablatoalpha \partial_z v_{i}}{\partial_{zz} v_{i}} \rangle_{\Omega}\\
&+\sum_{0<\alpha'<\alpha}\langle \nablatoalphaminstr w  \nablatoalphastr \partial_{zz} v_{i})
, \frac{\nablatoalpha\partial_z v_{i}}{\partial_{zz} v_{i}} \rangle_{\Omega}.
\end{align*}
Only the first term is new compared to b), we get
\begin{align*}
\langle \nablatoalpha w \partial_{zz} v_{i},\frac{\nablatoalpha \partial_z v_{i}}{\partial_{zz} v_{i}} \rangle_{\Omega} = \langle \nablatoalpha w,\nablatoalpha \partial_z v_{i}\rangle_{\Omega}
= \langle \nablatoalpha (\partial_x v_{1}+\partial_y v_{2}),\nablatoalpha v_{i}\rangle_{\Omega}
\end{align*}
and this is also the term for which 
forces us to use the more complicated multiplier.\\
d) Here we also split the expression into 
the highest order derivatives and a sum,
\begin{align*}
\langle \nablatoalpha (\partial_z v \cdot  \nablah & v_{i}+  \partial_z w \partial_z v_{i}), \frac{\nablatoalpha\partial_z v_{i}}{\partial_{zz} v_{i}} \rangle_{\Omega}\\
= & \langle \partial_z v \cdot \nablah  \nablatoalpha v_{i}- \nablatoalpha\divh v \; \partial_z v_{i}, \frac{\nablatoalpha\partial_z v_{i}}{\partial_{zz} v_{i}} \rangle_{\Omega}\\
& +\sum_{\alpha'<\alpha}\langle \nablatoalphaminstr \partial_z v \cdot \nablah  \nablatoalphastr  v_{i}-  \nablatoalphastr \divh v  \; \nablatoalphaminstr \partial_z v_{i}, \frac{\nablatoalpha\partial_z v_{i}}{\partial_{zz} v_{i}} \rangle_{\Omega}.
\end{align*}
The terms in the sum are once again easy to estimate,
\begin{multline*}
|\langle \nablatoalphaminstr \partial_z v \cdot \nablah  \nablatoalphastr  v_{i}-  \nablatoalphastr \divh v  \; \nablatoalphaminstr \partial_z v_{i}, \frac{\nablatoalpha\partial_z v_{i}}{\partial_{zz} v_{i}} \rangle_{\Omega}| \\ 
\leq c\eta\norm{v}_{H^s}\norm{\partial_z v}_{H^s} \norm{\nablatoalpha \partial_z v_{i}}_{L^2}.
\end{multline*}
For the other term we use a cancellation, for $i=1$ we obtain
\begin{multline*}
\langle \partial_z v \cdot \nablah  \nablatoalpha v_{1}- \nablatoalpha\divh v \; \partial_z v_{1}, \frac{\nablatoalpha\partial_z v_{1}}{\partial_{zz} v_{1}} \rangle_{\Omega}\\
=  \langle \partial_z v_2 \cdot \partial_y \nablatoalpha v_{1}- \nablatoalpha\partial_y v_{2} \cdot \partial_z v_{1}, \frac{\nablatoalpha\partial_z v_{1}}{\partial_{zz} v_{1}} \rangle_{\Omega}.
\end{multline*}
The first term can be estimated directly
\begin{align*}
|\langle \partial_z v_2 \cdot \partial_y \nablatoalpha v_{1}, \frac{\nablatoalpha\partial_z v_{1}}{\partial_{zz} v_{1}} \rangle_{\Omega}|
& \leq  c\eta\norm{\partial_z v_2}_{H^2}\norm{\nablatoalpha \partial_y v_{1}}_{L^2} \norm{\nablatoalpha \partial_z v_{1}}_{L^2}
\end{align*}
and for the second one we integrate by parts
\begin{align*}
\langle \nablatoalpha \partial_z v_{1} \; \partial_z v_{1}, \frac{\nablatoalpha\partial_y v_{2}}{\partial_{zz} v_{1}} \rangle_{\Omega}=
&-\langle \nablatoalpha \partial_z v_{1} \; \partial_y\partial_z v_{1} +\partial_z v_{1}\; \nablatoalpha \partial_y \partial_z v_{1} , \frac{\nablatoalpha v_{2} }{\partial_{zz} v_{1}} \rangle_{\Omega}\\
&+\langle \nablatoalpha \partial_z v_{1} \; \partial_z v_{1}, \partial_{zz}\partial_y v_{1}\frac{\nablatoalpha v_{2}}{(\partial_{zz} v_{1})^2} \rangle_{\Omega}.
\end{align*}
It follows
\begin{align*}
|\langle  \nablatoalpha \partial_z v_{1}\; \partial_z v_{1}, \frac{\nablatoalpha\partial_y v_{2}}{\partial_{zz} v_{1}} \rangle_{\Omega}|\leq 
& c\eta \norm{v_{2}}_{H^s}\norm{\partial_z v_{1}}_{H^3}( \norm{\partial_z v_{1}}_{H^s}+\norm{\partial_y \partial_z v_{1}}_{H^s})\\
& +c\eta^2 \norm{v_{2}}_{H^s} \norm{\partial_z v_{1}}_{H^2} \norm{\partial_{z}\partial_y v_{1}}_{H^3} \norm{\partial_z v_{1}}_{H^s}.
\end{align*}
So we have
\begin{align*}
|\langle \partial_z v & \cdot \nablah  \nablatoalpha v_{1}- \nablatoalpha\divh v \; \partial_z v_{1}, \frac{\nablatoalpha\partial_z v_{1}}{\partial_{zz} v_{1}} \rangle_{\Omega}| \leq c\eta \norm{v}_{H^s}\norm{\partial_z v}_{H^s}^2\\
& + c\eta\norm{\partial_z v}_{H^s}( \norm{v}_{H^s}+\norm{\partial_z v}_{H^s})\norm{\partial_y v_{1}}_{H^s} + c\eta^2\norm{v}_{H^s}\norm{\partial_z v}^2_{H^s}\norm{\partial_{z}\partial_y v_{1}}_{H^3}.
\end{align*}
For $i=2$ we proceed in the same way.\\
e) Last we have to handle the time 
derivative for this multiplier,
\begin{align*}
\langle (\nablatoalpha \ddt \partial_z v_{i},\frac{\nablatoalpha \partial_z v_{i}}{\partial_{zz} v_{i}} \rangle _{\Omega}
& =\frac{1}{2} \ddt \norm{\frac{\nablatoalpha \partial_z v_{i}}{\sqrt{\partial_{zz} v_{i}}}}_{L^2}^2- \frac{1}{2} \langle (\nablatoalpha \partial_z v_{i})^2, \ddt \frac{1}{\partial_{zz} v_{i}} \rangle_{\Omega}
\end{align*}
with
\begin{align*}
\ddt \frac{1}{\partial_{zz} v_{i}} = &
\frac{v \cdot \partial_{zz} \nablah  v_{i}
+ \partial_{zz} v \cdot \nablah  v_{i}}{(\partial_{zz} v_{i})^2}
- \frac{A_{\varepsilon,i}\partial_{zz} v_{i} }{(\partial_{zz} v_{i})^2}\\
& +  \frac{\partial_{zz} w \partial_{z} v_{i}
+ w \partial_{zzz} v_{i}+2\partial_{z} v \cdot \nablah \partial_z v_{i} 
+ 2 \partial_z w \partial_{zz} v_{i}}{(\partial_{zz} v_{i})^2}.
\end{align*}
The first fraction will cancel 
with terms obtained in b) and c), 
and the lengthy expression inherits 
only functions which are in $L^{\infty}(\Omega)$. 
This yields
\begin{multline*}
\norm{ \frac{\partial_{zz} v \cdot \nablah  v_{i}
+ \partial_{zz} w \partial_{z} v_{i}
+  2\partial_{z} v \cdot \nablah \partial_z v_{i} 
+ 2 \partial_z w \partial_{zz} v_{i}}{(\partial_{zz} v_{i})^2}}_{L^{\infty}}\\
\leq c\eta^2 \norm{\partial_{z} v}_{H^3}(\norm{v}_{H^3}+\norm{\partial_{z} v}_{H^3}).
\end{multline*}
Finally, for $\langle (\nablatoalpha \partial_z v_{i})^2,  \frac{A_{\varepsilon,i}\partial_{zz} v_i}{(\partial_{zz} v_{i})^2} \rangle_{\Omega}$ we obtain
\begin{multline*}
\langle (\nablatoalpha \partial_z v_{i})^2,  \frac{\partial_{xx}\partial_{zz} v_i}{(\partial_{zz} v_{i})^2} \rangle_{\Omega}\\
=-\langle 2(\nablatoalpha \partial_z v_{i})( \partial_{x}\nablatoalpha \partial_{z} v_{i}),  \frac{\partial_{x}\partial_{zz}v_i}{(\partial_{zz} v_{i})^2} \rangle_{\Omega}
+\langle (\nablatoalpha \partial_z v_{i})^2,  \frac{(\partial_{x}\partial_{zz} v_i)^2}{(\partial_{zz} v_{i})^3} \rangle_{\Omega}.
\end{multline*}
As in a) we get
\begin{align*}
|\langle 2(\nablatoalpha \partial_z v_{i})( \partial_{x}\nablatoalpha \partial_{z} v_{i}),  \frac{\partial_{x}\partial_{zz}v_i}{(\partial_{zz} v_{i})^2} \rangle_{\Omega}|
& \leq c\eta^2 \norm{\nablatoalpha \partial_z v_{i}}_{L^2}\norm{\partial_{x}\partial_{z}v_i}_{H^s}^{7/4} \norm{\partial_{z}v_i}_{H^1}^{1/4}
\end{align*}
and from Lemma \ref{lemma:lowreg} follows
\begin{align*}
|\langle (\nablatoalpha \partial_z v_{i})^2,  \frac{(\partial_{x}\partial_{zz} v_i)^2}{(\partial_{zz} v_{i})^3} \rangle_{\Omega}|
& \leq \eta^3 \norm{\nablatoalpha \partial_z v_{i}}^2_{L^2} \norm{\partial_{x}\partial_{zz} v_i}^2_{L^{\infty}}\\
& \leq c\eta^3 \norm{\nablatoalpha \partial_z v_{i}}^2_{L^2} \norm{\partial_{x}\partial_{z} v_i}^{3/2}_{H^3}\norm{\partial_{z} v_i}^{1/2}_{H^2}.
\end{align*}
Combining the two estimates yields
\begin{align*}
|\langle (\nablatoalpha \partial_z v_{i})^2,  \frac{\partial_{xx}\partial_{zz} v_i}{(\partial_{zz} v_{i})^2} \rangle_{\Omega}|
\leq & c\eta^2 \norm{\nablatoalpha \partial_z v_{i}}_{L^2} \norm{ \partial_{x}\nablatoalpha \partial_{z} v_{i}}_{L^2}  \norm{\partial_{x}\partial_{z}v_i}_{H^3}^{3/4} \norm{\partial_{z}v_i}_{H^2}^{1/4}\\
& +c\eta^3 \norm{\nablatoalpha \partial_z v_{i}}^2_{L^2} \norm{\partial_{x}\partial_{z} v_i}^{3/2}_{H^3}\norm{\partial_{z} v_i}^{1/2}_{H^2}
\end{align*}
and similarly we get
\begin{align*}
|\langle (\nablatoalpha \partial_z v_{i})^2,  \frac{\partial_{yy}\partial_{zz} v_i}{(\partial_{zz} v_{i})^2} \rangle_{\Omega}|
\leq & c\eta^2 \norm{\nablatoalpha \partial_z v_{i}}_{L^2} \norm{ \partial_{y}\nablatoalpha \partial_{z} v_{i}}_{L^2}  \norm{\partial_{y}\partial_{z}v_i}_{H^3}^{3/4} \norm{\partial_{z}v_i}_{H^2}^{1/4}\\
& +c\eta^3 \norm{\nablatoalpha \partial_z v_{i}}^2_{L^2} \norm{\partial_{y}\partial_{z} v_i}^{3/2}_{H^3}\norm{\partial_{z} v_i}^{1/2}_{H^2}.
\end{align*}
\end{proof}
Now we can give the estimate for the baroclinic mode.
\begin{lemma}\label{lemma:estimatebarotropic}
Let $s\geq3$, $\eta>1$, $\varepsilon>0$, 
$v_0\in H_{\per}^s(\Omega)$  with 
$\divh \overline v_0=0$, $ \partial_z v_0 \in H_{\per,\eta}^s(\Omega)$ 
and $v=\overline v+\tilde v$ be the solution to 
(\ref{eq:primequhorviscbar})-(\ref{eq:primequcoupling}) 
for $A=A_{\varepsilon}$ according to Corollary 
\ref{cor:horvisclocz}. Then the following estimate 
holds in $[0,T]$ for some $T>0$
\begin{align*}
\begin{split}
&\ddt \eta\norm{\partial_z v}^2_{H^s_{2\eta}}
+\frac{1}{2}\norm{\partial_{y}\partial_z v_{1}}^2_{H^s}
+\frac{1}{2}\norm{\partial_{x}\partial_z v_{2}}^2_{H^s}
+\frac{\varepsilon}{2}\norm{\partial_{x}\partial_z v_{1}}^2_{H^s}+\frac{\varepsilon}{2}\norm{\partial_{y}\partial_z v_{2}}^2_{H^s}\\
&\quad \leq c(\eta) (\norm{\overline v}_{H^s}^2+\norm{\partial_z v}_{H^s}^2+\norm{\overline v}_{H^s}^2\norm{\partial_z v}_{H^2}^4)+ c(\eta)(1+\varepsilon)(\norm{\partial_{z}v}_{H^s}^{10}+\norm{\partial_{z}v}_{H^s}^{6})\\
& \qquad +\frac{1}{2}\norm{\partial_{y} \overline v_{1}}_{H^s}+\frac{1}{2}\norm{\partial_{y} \overline v_{2}}_{H^s}.
\end{split}
\end{align*}
\end{lemma}
\begin{proof}
For $\nablatoalpha=(\partial_x,\partial_y)^{\alpha}$ 
we multiply (\ref{eq:primequvz}) by 
$\frac{\nablatoalpha\partial_z v_i}{|\partial_{zz} v_{i}|}$,
\begin{multline*}
\langle \partial_t \nablatoalpha \partial_z v_{i},\frac{\nablatoalpha\partial_z v_i}{|\partial_{zz} v_{i}|} \rangle_{\Omega}
- \langle A_{\varepsilon,i}\nablatoalpha\partial_z v_{i},\frac{\nablatoalpha\partial_z v_i}{|\partial_{zz} v_{i}|} \rangle_{\Omega}\\
=-\langle \nablatoalpha(\partial_z v \cdot \nablah  v_{i} + \partial_z w \partial_z v_{i}+v \cdot \nablah  \partial_z v_{i}+ w \partial_{zz} v_{i}),\frac{\nablatoalpha\partial_z v_i}{|\partial_{zz} v_{i}|} \rangle_{\Omega}.
\end{multline*}
With
\begin{align*}
|\langle \nablatoalpha (\partial_x v_{1}+\partial_y v_{2}),\nablatoalpha v_{1}\rangle_{\Omega}|
& \leq \norm{\nablatoalpha v_{2}}_{L^2} \norm{\nablatoalpha \partial_y v_{1}}_{L^2}
\end{align*}
follows from Lemma \ref{lemma:prepestimatebarotropic2} for $i=1$
\begin{align*}
&\sum_{|\alpha|=s}\frac{1}{2} \ddt \norm{\frac{\nablatoalpha \partial_z v_{1}}{\sqrt{|\partial_{zz} v_{1}|}}}_{L^2}^2
+\varepsilon\frac{1}{2\eta}\norm{\nablatoalpha\partial_{x} \partial_z v_{1}}^2_{L^2}+\frac{1}{2\eta}\norm{\nablatoalpha\partial_{y} \partial_z v_{1}}^2_{L^2} \\
& \quad\leq \norm{v_{2}}_{H^s} \norm{\partial_y v_{1}}_{H^s}+ c\eta \norm{v}_{H^s}\norm{\partial_z v}_{H^s}^2+c\eta^2 \norm{\partial_{z} v}_{H^s}^3(1+\norm{v}_{H^s})\\
& \qquad +c\eta\norm{\partial_z v}_{H^s}(\norm{\partial_z v}_{H^s}+\norm{v}_{H^s})\norm{\partial_y v_{1}}_{H^s}+c\eta^2 \norm{v}_{H^s}\norm{\partial_z v}_{H^s}^2 \norm{\partial_{z}\partial_y v_{1}}_{H^s}\\
& \qquad +c \eta^2\norm{\partial_{z} v}_{H^{s}}^{5/4}(\varepsilon\norm{\partial_{x}\partial_{z} v_{1}}_{H^{s}}^{7/4}+ \norm{\partial_{y}\partial_{z} v_{1}}_{H^{s}}^{7/4})\\
& \qquad +c\eta^3 \norm{\partial_{z} v}^{3/2}_{H^s}(\varepsilon\norm{\partial_{x}\partial_{z} v_1}^{3/2}_{H^s}+\norm{\partial_{y}\partial_{z} v_1}^{3/2}_{H^s}).
\end{align*}
We use $v=\tilde v+\overline v$, where 
$\tilde v$ is average free in the vertical 
direction, and Poincar\'e's inequality to get
\begin{align*}
\norm{\partial_y v}_{H^s(\Omega)}	&\leq \norm{\partial_y \tilde v}_{H^s(\Omega)}+\norm{\partial_y \overline v}_{H^s(\Omega)}
			  \leq c\norm{\partial_y \partial_z \tilde v}_{H^s(\Omega)}+\sqrt{2h}\norm{\partial_y \overline v}_{H^s(G)}.
\end{align*}
By this and by applying Young's inequality a couple of times follows
\begin{align*}
&\sum_{|\alpha|=s} \eta \ddt \norm{\frac{\nablatoalpha \partial_z v_{1}}{\sqrt{|\partial_{zz} v_{1}|}}}_{L^2}^2
+\varepsilon\norm{\nablatoalpha\partial_{x} \partial_z v_{1}}^2_{L^2}+\norm{\nablatoalpha\partial_{y} \partial_z v_{1}}^2_{L^2}\\
&\quad \leq c(\eta) (\norm{\overline v}_{H^s}^2+\norm{\partial_z v}_{H^s}^2+\norm{\overline v}_{H^s}^2\norm{\partial_z v}_{H^2}^4)+ c(\eta)(1+\varepsilon)(\norm{\partial_{z}v}_{H^s}^{10}+\norm{\partial_{z}v}_{H^s}^{6})\\
&\qquad +\frac{1}{2}\norm{\partial_{y} \overline v_{1}}^2_{H^s}+\frac{1}{2}\norm{\partial_{y}\partial_z v_{1}}^2_{H^s}+\frac{\varepsilon}{2}\norm{\partial_{x}\partial_{z} v_1}^{2}_{H^s}.
\end{align*}
For $i=2$ we can proceed in the same way. 
Together with Lemma \ref{lemma:prepestimatebarotropic1} 
and after applying Young's inequality some more 
times we obtain the stated estimate. 
\end{proof}

\subsubsection*{Local existence}
We finally have everything in place to 
give the proof of Theorem \ref{th:introd2}, i.e. to 
show the local in time well-posedness of 
(\ref{eq:primequhalfhorvisc}). Let us 
restate the Theorem using the $H_{\per,\eta}^s$-spaces.
\begin{theorem}
Let $s\geq3$ and $\eta>0$. Then for any 
$v_0\in H_{\per}^s(\Omega)$ with $\divh \overline v_0=0$ 
and $\partial_z v_0 \in H_{\per,\eta}^s(\Omega)$ 
there exists a $T>0$ and a unique strong solution 
$v$ to (\ref{eq:primequhalfhorvisc}) with
\begin{align*}
 v \in L^{\infty}((0,T), H_{\per}^{s}(\Omega)) \cap C^0([0,T], H_{\per}^{s-\kappa}(\Omega)),\; \partial_x v_2,\partial_y v_1 \in L^2((0,T),H_{\per}^{s}(\Omega))
\end{align*}
for all $\kappa\in(0,1)$ and $\partial_z v$ 
has the same regularity as $v$ with 
$\partial_z v(t)\in H_{\per,2\eta}^s(\Omega)$.
\end{theorem}

\begin{proof}
Let $n\in\N$ and $v^{(n)}=\overline v^{(n)}+\tilde v^{(n)}$ 
be the solution to 
(\ref{eq:primequhorviscbar})-(\ref{eq:primequcoupling}) 
for $A=A_{\varepsilon}$ with $\varepsilon=\frac1n$ 
according to Corollary \ref{cor:horvisclocz}.

Following Lemma \ref{lemma:estimatebarotropic} and 
\ref{lemma:baroclinic} there exists some $T_n>0$ such that 
\begin{align*}
&\ddt \left(\eta\norm{\partial_z v^{(n)}}^2_{H^s_{2\eta}}+\norm{\overline{v}^{(n)}}^2_{H^s} \right)\\
&\qquad\qquad +\frac{1}{4}\left(\norm{\partial_{y}\partial_z v^{(n)}_{1}}^2_{H^s}
+\norm{\partial_{x}\partial_z v^{(n)}_{2}}^2_{H^s}
+\norm{\partial_{y}\overline v^{(n)}_{1}}^2_{H^s}
+\norm{\partial_{x}\overline v^{(n)}_{2}}^2_{H^s}\right)\\
&\qquad\qquad +\frac1n\left(\frac{3}{2}\norm{\partial_{x}\partial_z v^{(n)}_{1}}^2_{H^s}
		+\frac{3}{2}\norm{\partial_{y}\partial_z v^{(n)}_{2}}^2_{H^s}
		+ \norm{\partial_x \overline{v}^{(n)}_{1}}_{H^s}^2
		+ \norm{\partial_y \overline{v}^{(n)}_{2}}_{H^s}^2\right)\\
& \leq c(\eta) \left[\norm{\overline{v}^{(n)}}_{H^{s}}^2+\norm{\partial_{z}v^{(n)}}_{H^s}^{2}+\left(\norm{\overline{v}^{(n)}}^{2}_{H^{s}}+\norm{\partial_{z}v^{(n)}}_{H^s}^{2}\right)^5\right]
\end{align*}
holds for $t\leq T_n$. Integration in time 
and Gronwall's inequality yield now
\begin{align*}
& \norm{\overline v^{(n)}(t)}^2_{H^s}+\eta \norm{\partial_z v^{(n)}(t)}^2_{H^s_{2\eta}}\\
& \quad + \frac14 \int_0^t \norm{\partial_y \overline{v}^{(n)}_{1}(r)}_{H^s}^2\!+\norm{\partial_x \overline{v}^{(n)}_{2}(r)}_{H^s}^2 
 \!+\norm{\partial_y \partial_z \tilde{v}^{(n)}_{1}(r)}_{H^s}^2\!+\norm{\partial_x \partial_z \tilde{v}^{(n)}_{2}(r)}_{H^s}^2\dr\\
& \leq \left[\left(1+\norm{\overline v_0}^2_{H^s}+\eta \norm{\partial_z v_0}^2_{H^s_{\eta}}\right)e^{-c(\eta)t}-1\right]^{-1/4}
\end{align*}
and so there exists some $T_a$ independent of $n$ such that 
\begin{align}\label{eq:estparhorvis}
\begin{split}
\norm{\overline v^{(n)}(t)}^2_{H^s}\!+ \eta\norm{\partial_z v^{(n)}(t)}^2_{H^s_{2\eta}}\! &+\frac14  \int_0^t  \norm{\partial_y \overline{v}^{(n)}_{1}(r)}_{H^s}^2\! +\norm{\partial_y \partial_z \tilde{v}^{(n)}_{1}(r)}_{H^s}^2 \\
& \qquad\quad  +\norm{\partial_x \overline{v}^{(n)}_{2}(r)}_{H^s}^2+\!\norm{\partial_x \partial_z \tilde{v}^{(n)}_{2}(r)}_{H^s}^2\dr\\
& \leq 2 \left(\norm{\overline v_0}^2_{H^s}+ \eta\norm{\partial_z v_0}^2_{H^s_{\eta}}\right)
\end{split}
\end{align}
for $t\leq \min\{T_a,T_n\}$. Additionally, we have to 
show that the existence time $T_n$ does not tend to $0$. 
The uniform bound for the norm is not sufficient, 
because of the dependence of $T_n$ on the condition 
$\frac{1}{2\eta}\leq \frac{1}{|v^{(n)}_{zz}(t)|}\leq 2\eta$, 
we need to control 
\begin{align*}
 \norm{\partial_{zz}v^{(n)}(t)-\partial_{zz}v^{(n)}(0)}_{L^{\infty}(\Omega)}.
\end{align*}
Using $v_{zz}^{(n)}\in L^{\infty}((0,T_n),H^2(\Omega))\cap W^{1,\infty}((0,T_n),L^2(\Omega))$ we obtain first that
\begin{align*}
 \partial_{zz}v^{(n)} \in L^{q}((0,T_n),H^2(\Omega))\cap H^{1,q}((0,T_n),L^2(\Omega))
\end{align*}
for all $1\leq q \leq\infty$, 
where $H^{1,q}$ denotes the Bessel potential space. This implies 
(see \cite[Theorem 2.4.1]{Tolksdorf2017} and \cite[Lemma 2.61]{DenkKaip2013})
\begin{align*}
 \partial_{zz}v^{(n)} \in H^{\kappa,q}((0,T_n),H^{2(1-\kappa)}(\Omega))
\end{align*}
for $\kappa\in[0,1]$ and both embeddings are continuous, thus
\begin{multline*}
\norm{\partial_{zz}v^{(n)}}_{H^{\kappa,q}((0,T_n),H^{2(1-\kappa)})}\\
	   \leq cT_n^{1/q}\left(\norm{\partial_{z}v^{(n)}}_{L^{\infty}((0,T),H^3)}+\norm{\partial_{z}v^{(n)}}_{L^{\infty}((0,T),H^3)}\norm{v^{(n)}}_{L^{\infty}((0,T),H^3)}\right).
\end{multline*}
Choosing now $\kappa=\frac18$ (any $\kappa <\frac14$ is possible) 
and $q=\frac{2}{\kappa}=16$ we get 
\begin{align*}
 H^{\kappa,q}((0,T),H^{2(1-\kappa)}(\Omega))\hookrightarrow C^{0,\alpha}([0,T],C^0(\Omega))
\end{align*}
by Sobolev embedding with $\alpha=\frac{1}{8}$ and
\begin{align*}
 \norm{\partial_{zz}v^{(n)}}_{C^{0,\alpha}([0,T],L^{\infty}(\Omega))} &\leq cT_n^{1/q}\left(1+\norm{v^{(n)}}_{L^{\infty}((0,T),H^3)}\right)\norm{\partial_{z}v^{(n)}}_{L^{\infty}((0,T),H^3)}.
\end{align*}
Therefore,
\begin{multline*}
 \norm{\partial_{zz}v^{(n)}(t)-\partial_{zz}v^{(n)}(0)}_{L^{\infty}(\Omega)}\\
  \leq c t^{\alpha}T_n^{1/q} \left(1+\norm{\overline v_0}^2_{H^s}+\eta  \norm{\partial_z v_0}^2_{H^s_{\eta}}\right)\left(\norm{\overline v_0}^2_{H^s}+ \norm{\partial_z v_0}^2_{H^s_{\eta}}\right)
\end{multline*}
and so there exists a $T_{b}$ independent of $n$ 
such that $\partial_{z}v^{(n)}(t)\in H^s_{2\eta}(\Omega)$ 
for $t\leq T_{b}$. Especially, (\ref{eq:estparhorvis}) 
holds for $t\leq\min\{T_{b},T_a\}$ and from the equations 
for $\overline v^{(n)}$, $\partial_z v^{(n)}$ we get uniform
bounds for $\partial_t\overline v^{(n)},\partial_t\partial_z v^{(n)}$
in $L^{\infty}((0,T),H^{s-2})$.
The convergence 
of a subsequence to a solution follows now by 
using Lemma \ref{le:aubinlions}.
\end{proof}

\begin{remark}
Masmoudi and Wong handled 
$\norm{\partial_{zz}v(t)-\partial_{zz}v(0)}_{L^{\infty}}$ 
for the primitive Euler equation in $2D$ by assuming 
enough regularity to have $\ddt \partial_{zz}v(t)\in L^{\infty}$ 
with a uniform estimate against the initial data, which gives
\begin{align*}
 \norm{\partial_{zz}v(t)-\partial_{zz}v(0)}_{L^{\infty}}\leq ct \norm{\partial_z v_0}_{H^s}.
\end{align*}
This approach would have forced us to take initial 
values $v_0,\partial_z v_0$ at least in $H^4(\Omega)$. 
Applying our method in their proof allows to lower 
their regularity assumptions on the data to 
$\partial_z v_0\in H^3((-1,1)\times(-h,h))$ instead 
of $\partial_z v_0\in H^4((-1,1)\times(-h,h))$.
\end{remark}

\subsection{Proof of Theorem \ref{th:introd3}}
\begin{proof}
Here we also consider for $n\in\N$ an equation 
with full, but anisotropic horizontal viscosity. 
However, deducing the energy estimate is much 
simpler than in the proof of Theorem \ref{th:introd2}. 
Let $v_n$ be the solution to
\begin{align*}
\begin{split}
\partial_t v_n + v_n \cdot \nablah  v_n + w_n \partial_z v_n - A'_{n} v_n + \nablah  p_n  & = 0 \qquad\mbox{in }(0,T)\times\Omega,  \\
\partial_z p_n  & = 0 \qquad\mbox{in }(0,T)\times\Omega,  \\
\divh v_n + \partial_z w_n & = 0 \qquad\mbox{in }(0,T)\times \Omega,\\
v(t=0)& = v_0  \!\!\qquad\mbox{in }\Omega
\end{split}
\end{align*}
with periodic boundary conditions 
imposed on $v_n$ and $p_n$ in the 
horizontal directions, $w_n(z=\pm h)=0$ and 
\begin{align*}
 A'_{n} =\left(\begin{matrix} \partial_{xx}+\frac1n \partial_{yy} & 0 \\ 0 & \partial_{yy}+\frac1n \partial_{xx} \end{matrix}\right).
\end{align*}
Theorem \ref{th:horviscloc} yields the 
existence of such a $v_n$ and multiplication 
of the equation in $H^s(\Omega)$ with $v_n$ 
gives because of $s\geq 3$
\begin{align*}
\begin{split}
\frac 12 \ddt \norm{v_n}^2_{H^s}+\norm{\partial_x v_{1n}}^2_{H^s} & + \frac1n \norm{\partial_y v_{1n}}^2_{H^s}+\frac1n \norm{\partial_x v_{2n}}^2_{H^s}+\norm{\partial_x v_{2n}}^2_{H^s}\\
		& \leq c \norm{v_n}^3_{H^s}+ c \norm{w_n}_{H^s} \norm{v_n}^2_{H^s}\\
		& \leq c \norm{v_n}^3_{H^s}+ c \norm{\divh v_n}_{H^s} \norm{v_n}^2_{H^s}\\
		& \leq c \norm{v_n}^3_{H^s}+ c \norm{v_n}^4_{H^s}+\frac12 \norm{\partial_x v_{1n}}^2_{H^s}+\frac12 \norm{\partial_y v_{2n}}^2_{H^s}.
\end{split}
\end{align*}
This already implies for $T$ sufficiently 
small that we have a uniform bound for 
$\norm{v_{n}}_{L^{\infty}((0,T),H^s)}$, $\norm{\partial_x v_{1n}}_{L^2((0,T),H^s)}$ 
and $\norm{\partial_y v_{2n}}_{L^2((0,T),H^s)}$, 
which yields again the convergence to a solution 
by using Lemma \ref{le:aubinlions}. 
\end{proof}


\begin{thebibliography}{99}

\bibitem{Brenier}
Y.~Brenier.
\newblock Homogeneous hydrostatic flows with convex velocity profiles.
\newblock {\em Nonlinearity}, 12(3):495--512, 1999.
\newblock \doi{10.1088/0951-7715/12/3/004}

\bibitem{Caoetal2015}
Ch.~Cao, S.~Ibrahim, K.~Nakanishi and E.~S.~Titi.
\newblock Finite-time blowup for the inviscid primitive equations of oceanic and atmospheric dynamics.
\newblock {\em Comm. Math. Phys.}, 337:473--482, 2015.
\newblock \doi{10.1007/s00220-015-2365-1}

\bibitem{CaoTiti2007}
Ch.~Cao and E.~S.~Titi.
\newblock Global well-posedness of the three-dimensional
viscous primitive equations of large scale ocean and
atmosphere dynamics.
\newblock {\em Annals of Mathematics}, 166:245--267, 2007.
\newblock \doi{10.4007/annals.2007.166.245}

\bibitem{CaoLiTiti2016}
Ch.~Cao, J.~Li and E.~S.~Titi.
\newblock Global well-posedness of the $3D$ primitive equations with only horizontal viscosity and diffusivity.
\newblock {\em Comm. Pure Appl. Math.}, 69:1492--1531, 2016.
\newblock \doi{10.1002/cpa.21576}

\bibitem{CaoLiTiti2017}
Ch.~Cao, J.~Li and E.~S.~Titi.
\newblock Strong solutions to the $3D$ primitive equations with only horizontal dissipation: Near $H^1$ initial data.
\newblock {\em J. Funct. Anal.}, 272(11):4606--4641, 2017.
\newblock \doi{10.1016/j.jfa.2017.01.018}

\bibitem{DenkKaip2013}
R.~Denk and M.~Kaip.
\newblock {\em General Parabolic Mixed Order Systems in $L^p$ and Applications}.
Second Edition.
\newblock Springer International Publishing Switzerland, 2013.
\newblock \doi{10.1007/978-3-319-02000-6}


\bibitem{GigaGriesHusseinHieberKashiwabara2020}
Y.~Giga, M.~Gries, M.~Hieber, A.~Hussein and
T.~Kashiwabara.
\newblock The hydrostatic Stokes semigroup and well-posedness 
of the primitive equations on spaces of bounded functions.
\newblock {\em J. Funct. Anal.}, 279(3):108561, 2020.
 \newblock \doi{10.1016/j.jfa.2020.108561}
 
 
 

\bibitem{HanNguyen2016}
D.~Han-Kwan and T.~T.~Nguyen.
\newblock Ill-Posedness of the hydrostatic Euler and singular Vlasov equations.
\newblock {\em Arch. Ration. Mech. Anal.},
221(3):1317--1344, 2016.
\newblock \doi{10.1007/s00205-016-0985-z}

\bibitem{Kobelkov2006}
G.~M.~Kobelkov.
\newblock Existence of a solution in the large for the $3D$ large-scale ocean dynamics equations.
\newblock {\em C. R. Math. Acad. Sci. Paris},
343(4):283--286, 2006.
\newblock \doi{10.1016/j.crma.2006.04.020}

\bibitem{Kukavica2014}
I.~Kukavica, N.~Masmoudi, V.~Vicol and T.~K.~Wong.
\newblock On the local well-posedness of the Prandtl and hydrostatic Euler equations with multiple monotonicity regions.
\newblock {\em SIAM J. Math. Anal.},
46(6):3865--3890, 2014.
\newblock \doi{10.1137/140956440}

\bibitem{Kukavica2011}
I.~Kukavica, R.~Temam, V.~Vicol and M.~Ziane.
\newblock Local existence and uniqueness for the hydrostatic Euler equations on a bounded domain.
\newblock {\em J. Differential Equations}, 250(3):1719--1746, 2011.
\newblock \doi{10.1016/j.jde.2010.07.032}

\bibitem{Ziane2007}
I.~Kukavica and M.~Ziane.
\newblock On the regularity of the primitive equations of the ocean.
\newblock {\em Nonlinearity} 20(12): 2739--2753, 2007.
\newblock \doi{10.1088/0951-7715/20/12/001}.	

\bibitem{LiTiti2016}
J.~Li and E.~S.~Titi.
\newblock Recent advances concerning certain class of geophysical flows.
\newblock In {\em Handbook of Mathematical Analysis in Mechanics of
Viscous Fluids}, Springer International Publishing, 2016.
\newblock \doi{10.1007/978-3-319-10151-4_22-1}

\bibitem{LiTiti2017}
J.~Li and E.~S.~Titi.
\newblock The primitive equations as the small aspect ratio
limit of the Navier-Stokes equations: rigorous 
justification of the hydrostatic approximation.
\newblock Preprint \href{https://arxiv.org/abs/1706.08885}{arXiv:1706.08885}, 2017.

\bibitem{Lionsetal1992}
J.~L.~Lions, R.~Temam and Sh.~H.~Wang.
\newblock New formulations of the primitive equations of
atmosphere and applications.
\newblock {\em Nonlinearity}, 5(2):237--288, 1992.
\newblock \doi{10.1088/0951-7715/5/2/001}

\bibitem{Lionsetal1992_b}
J.~L.~Lions, R.~Temam and Sh.~H.~Wang.
\newblock On the equations of the large-scale ocean.
\newblock {\em Nonlinearity}, 5(5):1007--1053, 1992.
\newblock \doi{10.1088/0951-7715/5/5/002}

\bibitem{Lionsetal1993}
J.~L.~Lions, R.~Temam and Sh.~H.~Wang.
\newblock Models for the coupled atmosphere and ocean.
({CAO} {I},{II}).
\newblock {\em Comput. Mech. Adv.}, 1:3--119, 1993.


\bibitem{Majda2003}
A.~Majda.
\newblock {\em Introduction to PDEs and waves for the
atmosphere and ocean}.
(Courant Lecture Notes in Mathematics vol 9).
\newblock Providence, RI:
American Mathematical Society, 2003.

\bibitem{MasmoudiWong2012}
N.~Masmoudi and T.~K.~Wong.
\newblock On the $H^s$ theory of hydrostatic Euler equations.
\newblock {\em Arch. Ration. Mech. Anal.},
204(1):231--271, 2012.
\newblock \doi{10.1007/s00205-011-0485-0}

\bibitem{Pedlosky1987}
J.~Pedlosky.
\newblock {\em Geophysical Fluid Dynamics}. Second Edition.
\newblock Springer, New York, 1987.
\newblock \doi{10.1007/978-1-4612-4650-3}    

\bibitem{Simon1987}
J.~Simon.
\newblock Compact sets in the space $L^p(0,T;B)$.
\newblock {\em  Ann. Mat. Pura Appl.},146(1):65--96, 1987.
\newblock \doi{10.1007/BF01762360}

\bibitem{Tolksdorf2017}
P.~Tolksdorf.
\newblock {\em On the $L^p$-theory of the Navier-Stokes equations on Lipschitz domains}.
\newblock Technische Universit\"at, Darmstadt, 2017.
\newblock \href{http://tuprints.ulb.tu-darmstadt.de/5960/}{http://tuprints.ulb.tu-darmstadt.de/5960/}

\bibitem{Triebel}
H.~Triebel.
\newblock {\em Theory of Function Spaces}.
\newblock (Reprint of 1983 edition)
Springer AG, Basel, 2010.
\newblock \doi{10.1007/978-3-0346-0416-1}

\bibitem{Vallis2006}
G.~K.~Vallis.
\newblock {\em Atmospheric and Oceanic Fluid Dynamics}.
Second Edition.
\newblock Cambridge Univ. Press, 2006.

\bibitem{WashingtonParkinson1986}
W.~M.~Washington and C.~L.~Parkinson.
\newblock {\em An Introduction to Three Dimensional Climate
Modeling}. Second Edition. 
\newblock University Science Books, 2005.

\bibitem{Wong2015}
T.~K.~Wong.
\newblock Blowup of solutions of the hydrostatic Euler equations.
\newblock {\em Proc. Amer. Math. Soc.}: 143: 1119--1125, 2015.
\newblock \doi{10.1090/S0002-9939-2014-12243-X}
\end{thebibliography}
\end{document}